\documentclass[11pt]{article}\usepackage[]{graphicx}\usepackage[]{xcolor}
\makeatletter
\def\maxwidth{ %
  \ifdim\Gin@nat@width>\linewidth
    \linewidth
  \else
    \Gin@nat@width
  \fi
}
\makeatother

\definecolor{fgcolor}{rgb}{0.345, 0.345, 0.345}

\usepackage{framed}
\makeatletter
 {\par\unskip\endMakeFramed%
 \at@end@of@kframe}
\makeatother

\definecolor{shadecolor}{rgb}{.97, .97, .97}
\definecolor{messagecolor}{rgb}{0, 0, 0}
\definecolor{warningcolor}{rgb}{1, 0, 1}
\definecolor{errorcolor}{rgb}{1, 0, 0}
\usepackage{alltt}

\usepackage{geometry,latexsym,graphics,authblk,color,amsmath,amsfonts,amssymb,amsbsy,amsthm,natbib,enumerate,url,multirow,threeparttable, epigraph,mathtools,bm}

\usepackage{array, booktabs}
\usepackage{caption}

\usepackage[skins]{tcolorbox}
\usepackage[colorlinks,
            linkcolor=blue,
            anchorcolor=blue,
            citecolor=blue]{hyperref}
\usepackage{float}
\usepackage{tikz-qtree}
\tikzset{every tree node/.style={minimum width=2em,draw,circle},
         blank/.style={draw=none},
         edge from parent/.style=
         {draw,edge from parent path={(\tikzparentnode) -- (\tikzchildnode)}},
         level distance=1.5cm}
\usepackage{setspace,comment}
\newtheorem{thm}{Theorem}
\newtheorem{deff}{Definition}
\newtheorem{lem}{Lemma}

\newtheorem{prop}{Proposition}

\newtheorem{assumption}{Assumption}

\newtheorem{example}{Example}

\newtheorem{fact}{Fact}

\DeclareMathOperator*{\argmin}{argmin}
\DeclareMathOperator*{\argmax}{argmax}

\newcommand\R{{\sf I\kern-0.1em R}}
\newcommand{\mathsym}[1]{{}}

\newcommand{\tabincell}[2]{\begin{tabular}{@{}1@{}}2\end{tabular}}

\IfFileExists{upquote.sty}{\usepackage{upquote}}{}
\IfFileExists{upquote.sty}{\usepackage{upquote}}{}
\IfFileExists{upquote.sty}{\usepackage{upquote}}{}
\begin{document}
\title{Full characterization of core for nonlinear optimization games}
\author[1]{Donglei Du\thanks{\tt ddu@unb.ca}}
\author[2]{Qizhi Fang\thanks{\tt qfang@ouc.edu.cn}}
\author[2]{Bin Liu\thanks{\tt binliu@ouc.edu.cn}}
\author[2]{Tianhang Lu\thanks{\tt tlu@ouc.edu.cn}}
\author[3]{Chenchen Wu\thanks{\tt wu\_chenchen\_tjut@163.com}}

\affil[1]{Faculty of Management, University of New Brunswick, Fredericton, New Brunswick, Canada, E3B 5A3}
\affil[3]{Department of Mathematics, Ocean University of Qingdao, Qingdao 266003, P. R. China} 
\affil[4]{College of Science, Tianjin University of Technology, Tianjin, P. R. China}

\maketitle

\begin{abstract}
We fully characterize the core of a broad class of nonlinear games by identifying a suitable relaxation for inherent nonlinearity, directly generalizing the linear frameworks in the literature. This characterization significantly expands the scope of cooperative games that can be analyzed and contributes to the literature on games induced from optimization models. We apply these insights to not only establish connections with and provide new insights on classical models but also solve new games untamed in the existing literature, including combinatorial quadratic and ratio games such as portfolio, maximum cut, matching, and assortment games. These results are further extended to more general models and also the approximate core. 
\end{abstract}


\tableofcontents
\section{Introduction}\label{sec:intro}

The central contribution of this work is Theorem~\ref{thm:nonlinear-packing-game} in Section~\ref{subsec:nl-game-main-result}, which fully characterizes the core for the following \emph{nonlinear} optimization game $\left([n]:=\{1,\dots, n\}, \nu_{X, f}\right)$ with $n$ agents, where the characteristic function $\nu_{X, f}$, defined over the Boolean lattice $\mathbb{B}^n:=\{0,1\}^n$, arises as the optimal value of the following nonlinear parametric program:
\begin{eqnarray}
\label{game:packing-subadditive}\mathbb{B}^n\ni w\mapsto \nu_{X, f}(w)&:=&\max\limits_{x\in X} \left[f(x): Ax\le w\right]\in  \mathbb{R},
\end{eqnarray}
where $X\subseteq \mathbb{R}^m_{+}$ is a subset of the positive orthant of the $m$-dimension Euclidean space such that $\{\mathbf{0}_m, \mathbf{e}_1,\dots, \mathbf{e}_m\}\subseteq X$; $A\in \mathbb{B}^{n\times m}$ is a matrix with binary entries such that $A\mathbf{e}_j\ne \mathbf{0}_n, \forall j\in [m]$ (i.e., $A$ has no zero column); and $f: X\mapsto \mathbb{R}$ is a real-valued function with $f\left(\mathbf{0}_m\right)=0$. Functions satisfying the last property will be called \emph{grounded}.

The \emph{core} of the (revenue) game $\nu_{X, f}(w)$ is defined as follows~\citep{gillies1953some,shapley1967balanced}: 
\begin{eqnarray*}
\text{core}\left(\nu_{X, f}\right)&:=&\left\{x\in \mathbb{R}^n: \mathbf{1}_n^Tx=\nu_{X, f}\left(\mathbf{1}_n\right), a^Tx\ge \nu_{X, f}(a), \forall a\in \mathbb{B}^n\right\}. 
\end{eqnarray*}
In the above,  $\mathbf{0}_n$ and $\mathbf{1}_n$ are the $n$-denominational all-zero and all-one vectors, respectively. For cost games, their cores are defined analogously by reversing the inequalities therein~\footnote{In this work, we promote the Boolean vector format over the set format to define concepts such as core. It is the authors' belief that the former provides less mental hurdle to establish more direct and transparent connections to the optimization theory (such as mathematical programming and integer programming) at large.}. 

A core is \emph{fully} characterized if (i) a sufficient and necessary condition for core non-emptiness is available; and moreover, (ii) whenever the core is non-empty, there exists a well-defined dual program whose optimal solutions corresponds bijectively with the core members.   

The well-known Bondareva-Shapley theorem~\citep{bondareva1963some,shapley1967balanced} therefore fully characterizes the core of any game $\mu$. However, games are usually not given in explicit formulas, instead they arise from practical optimization problems, such as (\ref{game:packing-subadditive}), resulting in the so called optimization games. Therefore how the Bondareva-Shapley theorem manifests itself in terms of the defining optimization problems has evolved into a rich and extensive research field with vast applications in economics, operations research, computer science and industry engineering and more recently in machine learning and artificial intelligence. 

While a significant portion of prior research has concentrated on specific optimization game instances using ad hoc methods (which will not be expatiated here; but interested readers are referred to surveys, such as \cite{Curiel1997,borm2001operations,nagarajan2008game,dror2011survey} and references therein, for further information), efforts have also been put into identifying simple yet powerful sufficient conditions or even characterizations of core non-emptiness for broad classes of optimization games, including linear programming games~\citep{owen1975core,granot1986generalized,samet1984core}, linear (binary) combinatorial optimization games~\citep{deng1999algorithmic,faigle2000core,deng2000totally,bilbao2012cooperative}, linear (non-binary) integer programming games~\citep{caprara2010new}, convex programming games~\citep{kalai1982generalized,dubey1984totally}, nonlinear (binary) integer programming games~\citep{liu2016computing}, lattice programming games~\citep{Topkis1998}, conic linear programming games~\citep{puerto2022conic}, and stochastic programming games~\citep{chen2009stochastic,zhang2009cost,uhan2015stochastic,chen2016duality}, among others.

The majority of these systematic studies of optimization games provide only sufficient conditions for core non-emptiness; and whenever available, the dual programs involved do not correspond bijectively with the core. Except for specific games, such as the facility location game~\citep{goemans2004cooperative}, and the matroid base game~\citep{nagamochi1997complexity}, among others, full characterizations of core non-emptiness for broad classes of optimization games are known only for \emph{linear} games, including the linear programming games~\citep{samet1984core}, and the (binary) integer linear programming games~\citep{deng1999algorithmic}. 

This status begs for the pursuit on identifying broad classes of \emph{nonlinear} (continuous or discrete) optimization games whose core non-emptiness can be \emph{fully} characterized as in (\ref{game:packing-subadditive}), subject to a balance between model generality with practical utility easiness.

The motivations for considering the game (\ref{game:packing-subadditive}) are of both practical and theoretical significance due to the introduction of nonlinearity into both the objective function and the constraints.

Practically, nonlinear combinatorial optimization has emerged as a compelling new research area within the field of combinatorial optimization, garnering significant attention. This domain integrates discrete choices with nonlinearity to identify optimal solutions. Consequently, it focuses on optimization problems where the objective function and/or constraints are nonlinear, and the feasible solutions are discrete combinatorial objects. By explicitly introducing nonlinearity, nonlinear combinatorial optimization affords more modelling flexibility and generalizes its linear counterpart, often leading to more intricate problems that demand new solution techniques. Notable examples include submodular optimization, discrete Newton methods, discrete DC program, primal-dual methods with convex relaxation, among others~\citep{murota2003discrete,li2006nonlinear,du2019nonlinear,belotti2013mixed}.
 
Theoretically, a prominent technique in previous \emph{linear} combinatorial game analysis involves relaxing the original integer linear program to its natural linear program relaxation to leverage linear program duality theory. However, for nonlinear games (both continuous and discrete), there is no `natural' relaxation; instead various relaxations are possible. Finding the right relaxation to fully characterize the core of our nonlinear games is one of the main challenges of this work.

While details will be provided in Section~\ref{sec:game:packing-subadditive}, it is worth highlighting the key differences between our techniques and results for our nonlinear games and those of \citep{samet1984core,deng1999algorithmic,faigle2000core,deng2000totally} for the linear games. From a technical standpoint, finding an appropriate relaxation of nonlinear game is non-trivial, unlike the linear case where the natural LP relaxation suffices. One may be tempted to adopt the well-known Lov\'{a}sz extension (for cost game) or multilinear extension (for revenue game) which have been widely used and extremely successful in the design and analysis of both exact and approximation algorithm in submodular optimization. However, this is undesirable for several reasons. Take the revenue game for example. First, for linear games, one implied but key property for full characterization of core non-emptiness is that the relaxed game has a non-empty core, which cannot be guaranteed in the relaxed games when the relaxation is the multilinear extension. Second, it is desirable that relaxed problem should be polynomially solvable so we can find a core member efficiently whenever the core is non-empty. This is impossible via the multilinear extension even the objective function is submodular, because then the relaxed problem involves DR-submodular maximization, an NP-hard problem in general. Third, our assumption on the objective function will be weaker than submodularity, which makes it even more untenable to the multilinear extension. Analogously, these concerns persist for cost games with non-trivial constraints when we use Lov\'{a}sz extension. 

It turns out for the purpose of fully characterizing the non-emptiness of the core of our nonlinear games, a new relaxation is needed (See Section~\ref{subsec:nl-game-main-result} for details.), which also necessitates emphasizing the differences between two concepts, relaxation vs extension, which have mostly been used interchangeably in the past up to this point (See Section~\ref{subsec:relax-vsa-ext} for details.). Moreover, our relaxation is non-oblivious in the sense that it depend on the constraints, in contrast to existing relaxations (such as the Lov\'{a}sz and multilinear extensions) which are oblivious.

The rest of this work is organized as follows, along with further contributions of this work. After some preliminaries in Section~\ref{sec:Preliminaries}, we provide a full characterization (Theorem~\ref{thm:nonlinear-packing-game}) for the core non-emptiness of the game (\ref{game:packing-subadditive}), assuming that the objective set function $f$ is \emph{individually subadditive} as defined in (\ref{eq:ind-sub-rhs}) in Section~\ref{sec:game:packing-subadditive}. 

Then two equivalent characterizations (Theorems~\ref{corr:packing-subadditive-intermediate-sandwich-equi}-\ref{thm:core_nonempty}) are offered by introducing two new games which bound the original game from above and below, both of which have a linear objective function in Section~\ref{subsec:equi-characterization-intermediate-game}. These alternative characterizations provides new quantitative (rather than qualitative) perspective on how our main results and games are related to existing linear games studied in previous literature. We then extend full characterizations to the cores of more general models in Section~\ref{sec:further-results}, including covering games (Theorem~\ref{thm:nonlinear-covering-game}) and partition games (Theorem~\ref{thm:nonlinear-partition-game}) in Section~\ref{sub:cover-partition-games}, games where $\text{cone}(X)$ is finitely generators in Section~\ref{subsec:gene-X-constraint-pack-game} (Theorem~\ref{thm:nonlinear-packing-const-X-generators}),  games with more general constraints matrix (Theorem~\ref{thm:nonlinear-packing-const-dep-w}), nonlinear non-binary integral games (Theorems~\ref{thm:nonlinear-packing-rhs-arb}) and nonlinear binary integral games where the objective function also depends on the agents (Theorems~\ref{thm:nonlinear-packing-game-obj-constraint}), and approximate core (Theorem~\ref{thm:appro-core}) in Section~\ref{subsec:appro-core}. We demonstrate the utility of our results in Section~\ref{sec:apps} by analyzing nonlinear versions of some of the classical linear games. This work also puts the concept of individually subadditiveness on the spotlight. To enable the utility of our model, it is critical to identify when a  function is individually subadditive. For this purpose, in Section~\ref{sec:subadd-operations-preservation}, we also investigate common operations that preserve individually subadditiveness.  Open questions worth further investigating are offered in Section~\ref{sec:conclusion}.

\section{Preliminaries and notations}\label{sec:Preliminaries}

Denote $\mathbb{R}^n$ as the $n$-dimension Euclidean space. Denote $\mathbb{R}^n_{+}$ and $\mathbb{R}^n_{++}$ as the non-negative and positive vectors in $\mathbb{R}^n$, respectively. Denote the join $x\vee y$ and meet $x\wedge y$ of two vectors $x,y\in \mathbb{R}^n$ as the coordinate-wise maximum and minimum, respectively. Hence $x^{+}:=\max\{x, 0\}=x\vee 0, \forall x\in\mathbb{R}$. For any $x\in \mathbb{R}^n$, denote the support set of $x$ as $\text{supp}(x)=\{i\in [n]: x_i\ne 0\}$. 

We use $A\odot B$ for the Hadamard product between any two matrices of the same dimension. We use $A\otimes B$ for the Kronecker product between any two matrices. For any positive integer $n\in\mathbb{N}$, we denote the ordered set $[n]:=\{1, \dots, n\}$. Denote $I_n\in\mathbb{R}^{n\times n}$ as the identify matrix. For any vector $b\in\mathbb{R}^n$, denote $\text{diag} (b):=(b_1, \dots, b_n)\in \mathbb{R}^{n\times n}$ as the diagonal matrix. 

Denote $\mathbb{B}^n$ as the $n$-dimension Boolean space. A binary vector $w\in \mathbb{B}^n$ is an indicator function of the set $S\subseteq [n]$ such that $w_i=1$ whenever $i\in S$. We use $\mathbf{0}_n, \mathbf{1}_n\in \mathbb{B}^n$ to represent the vectors with all zero's and all ones, respectively. For any $i\in [n], \mathbf{e}_i\in \mathbb{B}^n$ is the unit vector with all zeros except a one on the $i$-th coordinate. 

For any set $C\subseteq \mathbb{R}^n$, the \emph{indicator} function $\delta_C$ of $C$ is defined as $\delta_C(x)=0$, if $x\in C$ and $\delta_C(x)=+\infty$, if $x\notin C$. 

For function representation, we use two forms interchangeably, e.g., $f: \mathbb{R}^n\mapsto \mathbb{R}$ and $\mathbb{R}^n\ni x\mapsto f(x)\in \mathbb{R}$ both indicate a function $f$ whose domain and co-domain are respectively $\mathbb{R}^n$ and $\mathbb{R}$.

\section{Nonlinear games}\label{sec:game:packing-subadditive}

Under the assumption $A\in \mathbb{B}^{n\times m}$ and $f\left(\mathbf{0}_m\right)=0$, the nonlinear game $\nu_{X, f}$ in (\ref{game:packing-subadditive}) is evidently non-negative ($\nu_{X, f}\ge 0$), grounded ($\nu_{X, f}\left(\mathbf{0}_m\right)=0$), and monotone ($\forall w, u\in\mathbb{B}^n: w\ge u\implies \nu_{X, f}(w)\ge \nu_{X, f}(u)$). 

Note that our nonlinear game $\nu_{X, f}$ in (\ref{game:packing-subadditive}) is special case of the one in \citep[Equation (2)]{liu2016computing}, although their focus is in computing numerically good cost shares for these games, different from fully characterizing core non-emptiness, which is our main focus. They also claimed that $\nu_{X, f}$ is superadditive, whose proof is only valid for linear non-binary integer integer program, but for neither nonlinear integer program nor linear binary integer program. Therefore, it is still an open question to characterize the superadditiveness of $\nu_{X, f}$.

Our primary focus is to fully characterize core non-emptiness. To understand and present the characterization in Section~\ref{subsec:nl-game-main-result}, we first differentiate between \emph{relaxation} and \emph{extension}, two terms previously used interchangeably. We argue for a distinction between these terms to clarify their different bearings in Section~\ref{subsec:relax-vsa-ext}. This distinction is crucial for explaining the key differences between our techniques and results and those found in prior work, such as \citep{samet1984core,deng1999algorithmic}.

\subsection{Relaxation vs extension}\label{subsec:relax-vsa-ext}

As explained briefly in the introduction, when \( f(x) \) is nonlinear, the primary challenge lies in devising a relaxation for the nonlinear game~(\ref{game:packing-subadditive}).

\begin{deff}\label{def:relaxation-extension} (\textbf{Relaxation and extension})
Given a function $f: X_1\mapsto Y$, construct another function $F: X_2\mapsto Y$ by enlarging the domain from $X_1$ to $X_2\supseteq X_1$. Then 
\begin{enumerate}[(i)]
\item $F$ is a \textbf{relaxation} of $f$ if  
\begin{enumerate}[(a)]
\item either 
\[
 \forall x\in X_1: f(x)\le F(x)
\]
\item or
\[
\forall x\in X_1, f(x)\ge F(x)
\]
\end{enumerate}
\item $F$ is an \textbf{extension} if 
\[
\forall x\in X_1: F(x)=f(x)
\]
\end{enumerate}
\end{deff}
For simplicity, whenever there is no confusion, we use $F|_{X_1}\le f$ or $F|_{X_1}\ge f$ to represent relaxation, and $F|_{X_1}= f$ to represent extension. Therefore, extension is a stricter version of relaxation.

Based on these definitions, the standard convex relaxations in the literature, such as LP (linear programming) relaxation will be referred to as LP extension. This terminology would better align with established conventions in other mathematical fields, such as the Hahn-Banach Extension Theorem, the Kolmogorov Extension Theorem, and Urysohn’s Lemma in functional analysis; the Lov\'{a}sz extension and multilinear extension in submodular optimization; and the Carath\'{e}odory extension of a measure in measure theory, among others. Indeed, the concept of extension as a general principle and technique is ubiquitous across nearly all areas of mathematics, and for compelling reasons. Primarily, an extension preserves the properties inherited by a subset from the extended set, retaining their validity within the original subset. This is an exceptionally powerful idea in problem-solving; namely, addressing a more general problem that encompasses the original problem as a special case. 

We illustrate the different consequences between relaxation and extension in the context of optimization, formalized in the following straightforward fact (along with a simple proof in Appendix A for completeness). 
We consider maximization problems here, where relaxation means Definition~\ref{def:relaxation-extension}(i)(a). Analogously, Definition~\ref{def:relaxation-extension}(i)(b) should be used instead for minimization problems.
\begin{prop}\label{prop:relax-exten}
Given two subsets $X_1, X_2\subseteq \mathbb{R}^n$ such that $X_1\subseteq X_2$, let $F:X_2\mapsto \mathbb{R}$ be a relaxation of $f:X_1\mapsto \mathbb{R}$ such that $F|_{X_1}\ge f$. Consider the following two optimization problems, where the second is a relaxation of the first:
\begin{eqnarray*}
\max\limits_{x\in X_1} f(x)\le \max\limits_{y\in X_2} F(y).
\end{eqnarray*}
Let $x^*\in \argmax\limits_{x\in X_1} f(x)$ and $y^*\in\argmax\limits_{y\in X_2} F(y)$. 
\begin{enumerate}[(i)]
\item When $F$ is a \textbf{relaxation}, then 
\begin{enumerate}[(a)]
\item $f(x^*)=F(y^*)$ implies that $x^*\in\argmax\limits_{y\in X_2} F(y)$.
\item If $y^* \in X_1$, then  $F(y^*)\ge F(x^*)\ge f(x^*)\ge f(y^*)$, and hence it may happen that $y^*\notin \argmax\limits_{x\in X_1} f(x)$.
\item If $f(x^*)=F(y^*)$ and $y^* \in X_1$, then $F(y^*)=F(x^*)=f(x^*)=f(y^*)$, and hence $y^* \in\argmax\limits_{x\in X_1} f(x)$ and $x^*\in\argmax\limits_{y\in X_2} F(y)$.
\end{enumerate}
\item When $F$ is an \textbf{extension} of $f$, then 
\begin{enumerate}[(a)]
\item $f(x^*)=F(y^*)$ if and only if $x^*\in\argmax\limits_{y\in X_2} F(y)$.
\item If $y^* \in X_1$, then $F(y^*)=F(x^*)=f(x^*)=f(y^*)$, and hence $y^* \in\argmax\limits_{x\in X_1} f(x)$ and $x^*\in\argmax\limits_{y\in X_2} F(y)$.
\end{enumerate}
\end{enumerate}
\end{prop}

In view of Fact~\ref{prop:relax-exten}, a natural question arises: \emph{Is mere relaxation ever truly necessary and useful in its own right?} Surprisingly, as demonstrated in this work, we do require the strictly weaker notion of relaxation rather than extension to characterize the non-emptiness of the cores of the nonlinear game (\ref{game:packing-subadditive}) that we are now ready to explore.

\subsection{Full characterization of core}\label{subsec:nl-game-main-result}

For the game (\ref{game:packing-subadditive}), we impose the following assumptions: 
\begin{enumerate}[(a)]
\item $A\in \mathbb{B}^{n\times m}$ is a matrix with binary entries such that $A\mathbf{e}_j\ne \mathbf{0}_n, \forall j\in [m]$ (i.e., $A$ has no zero column);
\item $X\subseteq \mathbb{R}^m_{+}$ is a subset of the positive orthant of the $m$-dimension Euclidean space such that $\{\mathbf{0}_m, \mathbf{e}_1,\dots, \mathbf{e}_m\}\subseteq X$. 
\item $f: X\mapsto \mathbb{R}$ is a real-valued function that is \emph{grounded}, namely, $f\left(\mathbf{0}_m\right)=0$; and \emph{individually subadditive}, namely,  
\begin{equation}\label{eq:ind-sub-rhs}
f(x)\le (f(\mathbf{e}_1), \dots, f(\mathbf{e}_m))x, \forall x\in X.
\end{equation}
\qed
\end{enumerate}

While the first assumption (a) is a technical one to guarantee the maximum in (\ref{game:packing-subadditive}) is attained with a finite value, the last two assumptions need more explanations.

For (b), the set $X$ decides the modeling capacity of our model. One noteworthy feature of our model is that we are able to treat broad classes of (continuous, or discrete or mixed) nonlinear games in a unified way afforded by the flexibility of choosing different $X$, which can be linear or nonlinear. For example, $X=\mathbb{R}^m_{+}$ represents continuous games, such as those in \citep{samet1984core}; $X=\mathbb{B}^m$ represents integer (combinatorial) games, such as those in \citep{deng1999algorithmic}; $X=\mathbb{R}^k\times \mathbb{B}^{\ell}$ with $k+\ell=m$ represents mixed games; $X=\{x\in \mathbb{B}^m: \mathbf{1}_m^Tx\le k\}$, where $k\in \mathbb{N}$, represents games with a cardinality constraint; more generally, $X=\{x\in \mathbb{B}^m: Qx\le d\}$, where $Q\in \mathbb{R}^{\ell\times m}_{+}, d\in \mathbb{R}^{\ell}_{+}$ such that $Q\mathbf{e}_i\le d, \forall i\in [m]$, represents games with $\ell\in\mathbb{N}$ knapsack constraints; and even nonlinear constraints such as  $X=\{x\in \mathbb{B}^m: x^TQx\le 0\}$, where $Q\in \mathbb{R}^{\ell\times m}_{+}$ satisfies $\mathbf{e}_i^TQ\mathbf{e}_i\le 0, \forall i\in [m]$, among others. Therefore, our models and results advance existing literature by incorporating nonlinearity into both the objective function and the constraint. 

Note further that under the assumption that $\{\mathbf{0}_m, \mathbf{e}_1,\dots, \mathbf{e}_m\}\subseteq X$, we have the convex conic hull $\text{cone} (X)=\mathbb{R}^m_{+}$ where the unit vectors serve as a set of finite generators. This assumption will be relaxed to the more general setting where $\text{cone} (X)$ is generated by any finite set of generators (or equivalently a polyhedron cone) in Section~\ref{subsec:gene-X-constraint-pack-game}, which includes the facility location game in \citep{goemans2004cooperative} as a special case. 

For (c), the class of individually subadditive functions in (c) is broad. Evidently linear function is a special case. As another example, when $X=\mathbb{B}^m$, all grounded subadditive set functions are individually subadditive, and hence all grounded submodular set functions, as a subset of the grounded subadditive set functions, are individually subadditive. Section~\ref{sec:relationships} provides examples to illustrate that the inclusions are proper. In particular a quadratic set function $x^TQx$ is individually subadditive if and only if it is submodular, and the ratio of two monotone modular functions $\frac{c^Tx}{d_0+d^Tx}, c, d, d_0\ge 0$ is individually subadditive, as will be shown in Section~\ref{sec:subadd-operations-preservation} along with more examples. However, the class of individually subadditive functions heavily depend on the domain. For example when $X=[0,1]^m$, submodular function and individually subadditive funciton are independent concepts as will be also discussed in Section~\ref{sec:subadd-operations-preservation}.  These will allow us to handle nonlinear continuous and combinatorial games beyond linear objective function, such as  the aforementioned quadratic and ratio objective functions.

Analogously, a function $f$ is \emph{individually superadditive} if and only if $-f$ is \emph{individually subadditive}. Evidently linear functions are both \emph{individually subadditive} and \emph{individually superadditive}, and hence can also be equivalently called \emph{individually additive} or just additive.

Individual supadditiveness or individual subadditiveness is not a new concept per se, which has been often assumed as a prerequisite property for characteristic functions of cooperative games in the prior literature. However, our usage here is different. We are not assuming the (revenue) game $\nu_{X,f}$ is individually subadditive or superadditive (whose characterization is an open question); rather we assume the objective function $f$ in the parametric optimization model is individually subadditive, under which, we are seeking to characterize the core non-emptiness of the induced game $\nu_{X,f}$. This actually opens up some open questions worth further exploring: (i) characterize the individual supadditiveness or the stronger supadditiveness of the revenue game $\nu_{X,f}$, under appropriate property of $f$; and (ii) characterize the core non-emptiness of the revenue game $\nu_{X,f}$, when $f$ is not individually subadditive, such as individually supadditive instead. 

The right-hand-side linear function of (\ref{eq:ind-sub-rhs}) can be viewed as a relaxation (although not an extension in general as exemplified after the following definition), which is well-defined for any function $f: X\mapsto \mathbb{R}$, not just individually subadditive functions. We call this relaxation the \emph{basis-linear} relaxation.  
\begin{deff}\label{deff:basis-linear-relaxation} (\textbf{Basis-linear relaxation}) Given a function $f: X\mapsto \mathbb{R}$ with $f\left(\mathbf{0}_m\right)=0$, where $\left\{\mathbf{0}_m, \mathbf{e}_1,\dots, \mathbf{e}_m\right\}\subseteq X\subseteq \mathbb{R}^m_{+}$, its basis-linear relaxation $F$ is given as follows:
\[
\mathbb{R}^m_{+}\ni x\mapsto F(x):=(f(\mathbf{e}_1),\dots, f(\mathbf{e}_m))x\in\mathbb{R}
\]
\end{deff}
Evidently, we have $F\left(\mathbf{e}_i\right) =f\left(\mathbf{e}_i\right), \forall i\in [m]$. However, we want to emphasize that this basis-linear relaxation is merely a relaxation of $f: X\mapsto \mathbb{R}$ rather than an extension because in general we only have $F\left(\mathbf{e}_i\right) =f\left(\mathbf{e}_i\right), \forall i\in [m]$ and $F(x)|_X\gneq f(x)$. As a quick example of this fact, given $X=\mathbb{B}^2\ni (x_1, x_2)\mapsto f(x_1, x_2)=x_1+x_2-x_1x_2\mapsto \mathbb{R}$, then $f$ is individually subadditive, but its basis-linear relaxation function $\mathbb{R}^2\ni x\mapsto F(x_1, x_2)=x_1+x_2\in\mathbb{R}$ is merely a relaxation and not an extension because $F(1,1)=2>1=f(1,1)$. 

With the basis-linear relaxation, we introduce the following relaxed linear game along with its dual:
\begin{eqnarray}
\label{eq:NLP-relax-primal}
\mathbb{B}^n\ni w\mapsto \nu_{\mathbb{R}^m_{+}, F}(w)&:=&\max\limits_{x\in \mathbb{R}^m_{+}} \left[F(x): Ax\le w\right] \\
\label{eq:NLP-relax-dual}&\overset{\text{LP duality}}{=}&\min\limits_{y\in \mathbb{R}^n_{+}} \left[w^Ty: y^TA\ge (f(\mathbf{e}_1),\dots, f(\mathbf{e}_m))\right]\in\mathbb{R}
\end{eqnarray}

Now we are ready to present the main result of this work. 
\begin{thm}\label{thm:nonlinear-packing-game}
For the nonlinear game (\ref{game:packing-subadditive}) under assumptions (a-c), we have
\begin{enumerate}[(i)]
\item The core of the game $\nu_{X, f}$ is non-empty if and only if $\nu_{X, f}\left(\mathbf{1}_n\right)=\nu_{\mathbb{R}^m_{+}, F}\left(\mathbf{1}_n\right)$ as in (\ref{game:packing-subadditive}) and (\ref{eq:NLP-relax-primal}), respectively, by choosing $w=\mathbf{1}_n$.
\item The core of $\nu_{X, f}$, whenever non-empty, coincides with the set of optimal solutions of the dual LP as in (\ref{eq:NLP-relax-dual}) by choosing $w=\mathbf{1}_n$; namely
 \[
 \text{core} \left(\nu_{X, f}\right)=\argmin\limits_{y\in \mathbb{R}^n_{+}} \left[\mathbf{1}_n^Ty: y^TA\ge (f(\mathbf{e}_1),\dots, f(\mathbf{e}_m))\right]
 \]
\item For nonlinear combinatorial game where $X=\mathbb{B}^m$, whenever the core is non-empty, the LP, as in (\ref{eq:NLP-relax-primal}) by choosing $w=\mathbf{1}_n$, has an integer optimal solution. 
\end{enumerate}
\end{thm}
\begin{proof} 
We first show the following two sets are the same. 
\begin{eqnarray*}
\Omega_1&:=&\left\{y\in\mathbb{R}^n_{+}: a^Ty\ge \nu_{X, f}(a), \forall a\in\mathbb{B}^n\right\}\\
&=&\left\{y\in\mathbb{R}^n_{+}: a^Ty\ge f(x), \forall x\in \{x\in X: Ax\le a\}, \forall a\in \mathbb{B}^n\right\}\\
\Omega_2&:=&\left\{y\in\mathbb{R}^n_{+}: y^TA\ge (f(\mathbf{e}_1),\dots, f(\mathbf{e}_m))\right\}
\end{eqnarray*}

First note that both sets are unbounded non-empty closed convex polyhedrons due to the assumptions in the model. Because $f$ is real-valued function, large enough $y$ will guarantee $\Omega_1, \Omega_2\ne \emptyset$, where $\Omega_1 \ne \emptyset$ is due to the assumption that $f\left(\mathbf{0}_m\right)=0$, and $\Omega_2 \ne \emptyset$ is due to the assumption that $A\mathbf{e}_j\ne \mathbf{0}_n, \forall j\in [m]$ (otherwise, if $\exists j\in [m]: A\mathbf{e}_j=\mathbf{0}_n$, then in the program (\ref{game:packing-subadditive}), no constraints are imposed on $x_j$ corresponding with the zero column $j$, and hence can be chosen to make the maximization problem unbounded, implying infeasible primal, a contradiction to our model assumption). Another point worth mentioning is that, for general $X$, the number of linear constraints in $\Omega_1$ may be infinite. However we do not need to invoke the toolkit from semi-infinite programming theory because under our modeling setting, we will show that $\Omega_1=\Omega_2$ while the latter is described with only finite number of linear constraints; namely the semi-infinite $\Omega_1$ is reducible to $\Omega_2$ (See ~\citep[Section 3; Definition 3.1]{shapiro2009semi} for the definition of reducibility).

For one direction, choosing $a=A\mathbf{e}_j\in \mathbb{B}^n, j\in [m]$ implies that $x=\mathbf{e}_j\in \{x\in X: Ax\le a=A\mathbf{e}_j\}$ because $\{\mathbf{0}_m, \mathbf{e}_1,\dots, \mathbf{e}_m\}\subseteq X\subseteq \mathbb{R}^m_{+}$.
\begin{eqnarray*}
y\in \Omega_1&\overset{a=A\mathbf{e}_j, x=\mathbf{e}_j, j\in [m]}{\implies}& \left[ (A\mathbf{e}_j)^Ty\ge f(\mathbf{e}_j), j\in [m], y\ge 0 \right]\\
&\iff& \left[y^TA\ge (f(\mathbf{e}_1),\dots, f(\mathbf{e}_m)), y\ge 0\right] \implies y\in \Omega_2
\end{eqnarray*}
For the other direction, $\forall a\in \mathbb{B}^n, \forall x\in \{x\in X: Ax\le a\}$:
\begin{eqnarray*}
y\in \Omega_2&\implies& \left[a^Ty\overset{y\ge 0}{\ge} y^TAx\overset{x\ge 0}{\ge} (f(\mathbf{e}_1),\dots, f(\mathbf{e}_m))x=F(x) \ge f(x)\right]\\
&\implies& y\in \Omega_1,
\end{eqnarray*}
where we used the fact that $y\ge 0$ due to the monotonicity of the game $\nu$ (However, the nonnegativity of $y$ is not needed for partition game due to $Ax=w$ as in Theorem~\ref{thm:nonlinear-partition-game}). 

Then we have the following relationships:  
\begin{eqnarray*}
\nonumber \nu_{X, f}\left(\mathbf{1}_n\right)&\le &\min\limits_{y\in\mathbb{R}^{n}_{+}}\left\{\mathbf{1}_n^Ty: a^Ty\ge \nu_{X, f}(a), \forall a\in\mathbb{B}^n\right\}\\
\nonumber&=&\min\limits_{y\in\mathbb{R}^{n}}\left\{\mathbf{1}_n^Ty: y\in \Omega_1\right\}\overset{\Omega_1=\Omega_2}{=}\min\limits_{y\in\mathbb{R}^{n}} \left[\mathbf{1}_n^Ty: y\in \Omega_2\right]\\
\nonumber&\overset{\text{LP dual}}{=}&\max\limits_{x\in \mathbb{R}^m_{+}} \left[(f(\mathbf{e}_1),\dots, f(\mathbf{e}_m))x: Ax\le \mathbf{1}_n\right]\overset{(\ref{eq:NLP-relax-primal})}{:=}\nu_{\mathbb{R}^m_{+}, F}\left(\mathbf{1}_n\right),
\end{eqnarray*}

So the core of $\nu_{X, f}$ is nonempty if and only if $\nu_{X, f}\left(\mathbf{1}_n\right)=\nu_{\mathbb{R}^m_{+}, F}\left(\mathbf{1}_n\right)$. This proves (i). Moreover, $\Omega_1=\Omega_2$ implies (ii). Proposition~\ref{prop:relax-exten}(i)(a) implies (iii).  
\end{proof}

\subsection{Immediate applications}\label{subsec:immediate-applications}
To put our research into further perspective and also as immediate applications of our results, we highlight special cases of Theorem~\ref{thm:nonlinear-packing-game} from the literature; offer some immediate new results beyond existing results; and point out key distinctions between the linear games from the literature and our nonlinear games. 
\begin{itemize}
\item \citep[Theorem 4]{samet1984core} is a special case of Theorem~\ref{thm:nonlinear-packing-game} by fixing the continuous domain $X=\mathbb{R}^m_{+}$, namely, the linear programming game $\nu_{\mathbb{R}^m_{+},c^Tx}$, whose core is always non-empty because the condition (i) therein trivially holds (this is the point we raised earlier about the non-emptiness of the relaxed game, which is critical for all analysis to carry through for linear games), and what \citep[Theorem 4]{samet1984core} showed is (ii), namely, the core, whenever non-empty, can be fully characterized by the dual LP. 
\item On the other hand, \citep[Theorem 1]{deng1999algorithmic} is a special case of Theorem~\ref{thm:nonlinear-packing-game} by fixing the discrete domain $X=\mathbb{B}^m$, namely, the linear combinatorial game $\nu_{\mathbb{B}^m, c^Tx}$, whose core may be empty and what \citep[Theorem 1]{deng1999algorithmic} showed is both (i) and (ii) therein. Moreover, in this special case, (i) is also equivalent to the condition that the LP relaxation has an integer optimal solution. Note further that the game studied in \citep{deng1999algorithmic} is actually a special nonlinear game of (\ref{game:packing-subadditive}) with linear objective because the binary constraint $x\in \mathbb{B}^m$ is inherently nonlinear, e.g., $x\in \mathbb{B}^m \iff x\in X:=\left\{x\in \mathbb{R}^m_{+}: x_i(1-x_i)=0, i\in [m]\right\}$ and this $X$ just defined satisfies $\{\mathbf{0}_m, \mathbf{e}_1,\dots, \mathbf{e}_m\}\subseteq X$. 
\item As mentioned earlier, the flexibility afforded by choosing different $X$ in Theorem~\ref{thm:nonlinear-packing-game} is beyond the pure linear and integer linear games investigated in the previous literature. It also holds for mixed integer linear programming games where $X=\mathbb{R}^k_{+}\times \mathbb{B}^{\ell}$ with $k+\ell=m$. 

\item The Bondareva-Shapley theorem~\citep{bondareva1963some,shapley1967balanced}, in particular, \citep[Lemma 1]{shapley1967balanced} is a special case of partition version of Theorem~\ref{thm:nonlinear-packing-game} (i.e., Theorem~\ref{thm:nonlinear-partition-game}), as explained below. This connection seems to be unknown previously. Consequently, our model is general enough to imply the famous Bondareva-Shapley theorem~\citep{bondareva1963some,shapley1967balanced}. This is only possible by choosing $X$ appropriately, a luxury not afforded by the models of \citep[Theorem 4]{samet1984core} and~\citep[Theorem 1]{deng1999algorithmic} which are too special to imply the Bondareva-Shapley theorem. 

It is well-known that any grounded (revenue) game $\nu: \mathbb{B}^n\mapsto \mathbb{R}$  agree with its double concave conjugate (least affine majoriant) $\nu^{\star\star}$ on $\mathbb{B}^n$~\citep[Section 4.8]{stoer1970convexity}: 
\begin{eqnarray*}
\label{eq:biconjugate-max}\mathbb{B}^n\ni w\mapsto \nu(w)=\nu^{\star\star}(w)&:=&\max_{\lambda\in \mathbb{R}^{2^n}_{+}}\left[\sum_{a\in\mathbb{B}^n}\nu(a)\lambda(a): \sum_{a\in\mathbb{B}^n}a\lambda(a)=w, \sum_{a\in\mathbb{B}^n}\lambda(a)=1\right]\\ 
&\overset{\nu(\mathbf{0}_n=0)}{=}&\max_{\gamma\in \mathbb{R}^{2^n-1}_{+}}\left[\sum_{a\in\mathbb{B}^n\backslash \{\mathbf{0}_n\}}\nu(a)\gamma(a): \sum_{a\in\mathbb{B}^n\backslash \{\mathbf{0}_n\}}a\gamma(a)=w, \sum_{a\in\mathbb{B}^n\backslash \{\mathbf{0}_n\}}\gamma(a)\le 1\right]
\end{eqnarray*}
Note that $\nu(w)=\nu_{X, f}$ is a partition game in terms of the notations of Theorem~\ref{thm:nonlinear-partition-game}, where 
\begin{eqnarray*}
X&:=&\left\{\lambda\in \mathbb{R}^{2^n-1}_{+}: \sum_{a\in\mathbb{B}^n\backslash \{\mathbf{0}_n\}}\lambda(a)\le 1\right\}\\
f&:=&\sum_{a\in\mathbb{B}^n\backslash \{\mathbf{0}_n\}}\nu(a)\lambda(a)\\
A&\in&\mathbb{B}^{n\times (2^n-1)}: A_{i, a}=a_i, i\in [n], a\in \mathbb{B}^n\backslash \{\mathbf{0}_n\}
\end{eqnarray*}
Note further the relaxed game $\nu_{\mathbb{R}^{2^n-1}_{+}, f}$ of $\nu_{X, f}$ is just the \emph{totally balanced cover game} $\nu^{\text{tbc}}$ of $\nu$ introduced in \citep[(4-4)-(4-5)]{shapley1967balanced}. 
\begin{eqnarray*}
\mathbb{B}^n\ni w\mapsto \nu_{\mathbb{R}^{2^n-1}_{+}, f}=\nu^{\text{tbc}}(w)&:=&\max_{\lambda\in \mathbb{R}^{2^n-1}_{+}}\left[\sum_{a\in\mathbb{B}^n\backslash \{\mathbf{0}_n\}}\nu(a)\lambda(a): \sum_{a\in\mathbb{B}^n\backslash \{\mathbf{0}_n\}}a\lambda(a)=w\right]\\ 
&\overset{\text{dual}}{=}&\min_{y\in\mathbb{R}^n}\left[w^Ty: a^Ty \ge \nu(a), \forall a\in\mathbb{B}^n\backslash \{\mathbf{0}_n\}\right] 
\end{eqnarray*}

Now Theorem~\ref{thm:nonlinear-partition-game} can be invoked to imply the Bondareva-Shapley theorem~\citep[Lemma 1]{shapley1967balanced}.

\item A crucial distinction emerges in nonlinear combinatorial games where $X=\mathbb{B}^m$. Unlike linear games, the existence of an integer dual solution is only a necessary, not a sufficient, condition for core non-emptiness. Specifically, Proposition~\ref{prop:relax-exten}(i)(a) reveals a key implication: if $\nu_{\mathbb{R}^m_{+}, f}\left(\mathbf{1}_n\right)=\nu_{\mathbb{R}^m_{+}, F}\left(\mathbf{1}_n\right)$, then the linear program (LP) $\nu_{\mathbb{R}^m_{+}, F}\left(\mathbf{1}_n\right)$ possesses an integer optimal solution. However, the converse does not necessarily hold. Therefore, the existence of an integer optimal solution for $\nu_{\mathbb{R}^m_{+}, F}\left(\mathbf{1}_n\right)$ does not guarantee core non-emptiness, a critical difference from the results presented in \citep[Theorem 1]{deng1999algorithmic}. Here is one example to illustrate this failure with constraint matrix $A=I_2$, the 2-by-2 identity matrix. 
\begin{example}\label{exp:int-imlies-no-non-empty-core}
\[
\mathbb{B}^2\ni w\mapsto \nu_{\mathbb{B}^2,f}(w):=\max_{(x_1, x_2)\in\mathbb{B}^2}\{x_1+x_2-2x_1x_2: x_1\le w_1, x_2\le w_2\}\in\mathbb{R},
\]
with $\nu_{\mathbb{B}^2,f}(1,1)=1$ and optimal solutions $(x_1, x_2)=(1,0), (0,1)$. The LP relaxation is as follows: 
\[
\nu_{\mathbb{R}^2,F}(1,1)=\max_{(x_1, x_2)\in\mathbb{R}^2}\{x_1+x_2: 0\le x_1\le 1, 0\le x_2\le 1\}=2
\]
Because $\nu_{\mathbb{B}^2, f}(1,1)=1\ne 2=\nu_{\mathbb{R}^2, F}(1,1)$, Theorem~\ref{thm:nonlinear-packing-game} implies the core of this game $\nu$ is empty, which can also be checked directly. If $(u_1, u_2)$ is a core member of $\nu_{\mathbb{B}^2, f}$, then we have a contraction: $u_1\ge \nu_{\mathbb{B}^2, f} (1, 0)=1, u_2\ge \nu_{\mathbb{B}^2, f} (0, 1)=1\implies u_1+u_2\ge 2\ne 1=\nu_{\mathbb{B}^2, f} (1, 1)$. However, the LP relaxation has a unique integer optimal solution $(x_1, x_2)=(1,1)$ which are not optimal solutions of the original game when $w=(1,1)$, and hence showing Proposition~\ref{prop:relax-exten}(i-b).
\end{example}
\end{itemize}

\subsection{Equivalent characterizations}\label{subsec:equi-characterization-intermediate-game}

Theorem~\ref{thm:nonlinear-packing-game} characterizes the core non-emptiness by relating the original game 
$\nu_{X, f}$ with the relaxed anchor game $\nu_{\mathbb{R}^m_{+}, F}$. We now offer two more equivalent characterizations by introducing two new games and explore their relationships, so as to provide alternative perspectives on our main results and also transparentize the connections with existing results. 

Let $\overline{X}:=\{x\in X: f(x)=F(x)\}\ne\emptyset$ be the set of \emph{extension} points of $f$. Assumption (b) implies the non-emptiness of this set, and $\text{cone}(\overline{X})=\text{cone}(X)=\mathbb{R}^m_{+}$. As explained earlier, since $F$ is only a relaxation rather than an extension of $f$, we usually have proper inclusion $\overline{X}\subsetneq X$ and this is why we call points in $\overline{X}$ as \emph{extension} points where $f$ and $F$ agree: $F|_{\overline{X}}=f|_{\overline{X}}$. 

We introduce two new games below. The first of which $\nu_{X, F}$ in (\ref{game:packing-subadditive-intermediate-X-F}) is called the \emph{upper} game, which linearizes the objective function from $f$ to $F$ without changing the constraint $X$; and the second of which $\nu_{\overline{X}, F}=\overset{f|_{\overline{X}}=F|_{\overline{X}}}{=}\nu_{\overline{X}, f}$ in (\ref{game:packing-subadditive-intermediate-X-cap-Y-F}) is called the \emph{lower} game, which restricts the constraint from $X$ to $\overline{X}$, which also effectively linearizes the objective from $f$ to $F$ over $\overline{X}$. Hence both the upper and lower games are linear objective games with nonlinearity coming from the constraint alone. 
\begin{eqnarray}
\label{game:packing-subadditive-intermediate-X-F}
\mathbb{B}^n\ni w\mapsto\nu_{X, F}(w)&:=&\max\limits_{x\in X} \left[F(x): Ax\le w\right]\in \mathbb{R}\\
\label{game:packing-subadditive-intermediate-X-cap-Y-F}
\mathbb{B}^n\ni w\mapsto \nu_{\overline{X}, F}(w)&:=&\max\limits_{x\in \overline{X}} \left[F(x): Ax\le w\right]\overset{f|_{\overline{X}}=F|_{\overline{X}}}{=}\nu_{\overline{X}, f}(w):=\max\limits_{x\in \overline{X}} \left[f(x): Ax\le w\right]\in \mathbb{R}
\end{eqnarray}

Then evidently, the following orders hold among the four games:  
\begin{equation}\label{ineq:nonlinear}
\underbrace{\nu_{\mathbb{R}^m_{+},F}(w)}_{\text{anchor game}}\overset{X\subseteq \mathbf{R}^m_{+}}{\ge} \underbrace{\nu_{X,F}(w)}_{\text{upper game}}\overset{F\ge f}{\ge} \underbrace{\nu_{X,f}(w)}_{\text{original game}}\overset{X\supseteq \overline{X}}{\ge} \nu_{\overline{X},f}(w)\overset{f|_{\overline{X}}=F|_{\overline{X}}}{=}\underbrace{\nu_{\overline{X},F}(w)}_{\text{lower game}}, \forall w\in\mathbb{B}^n
\end{equation}
Now the original nonlinear game $\nu_{X,f}$ is bounded from above and below by two games $\nu_{X,F}$ (the upper game) and $\nu_{\overline{X},F}$ (the lower game), respectively. Both games have the same linear objective function $F$, while the former has the same domain $X$, and the latter has a smaller domain $\overline{X}$. Note that the anchor game $\nu_{\mathbb{R}^m_{+},F}$ is the common upper bound, which was used in Theorem~\ref{thm:nonlinear-packing-game} and will be used in all the upcoming characterizations. 

Theorem~\ref{thm:nonlinear-packing-game} fully characterizes the core via relating the original game with  the anchor game. Next we provide two more full characterizations: (i) one is via relating the upper game with the anchor game (Theorem~\ref{corr:packing-subadditive-intermediate-sandwich-equi}); and (ii) the other is via relating the lower game with the anchor game (Theorem~\ref{thm:core_nonempty}).

\subsubsection{Characterization via the upper game}
Going upstream from the original game $\nu_{X,f}$ in (\ref{ineq:nonlinear}), we have the following schematic relationship, indicating how new games are constructed from the original game:
\[
\nu_{X, f}\overset{f\mapsto F}{\implies} \nu_{X, F}\overset{X\mapsto \text{cone}(X)=\mathbf{R}^m_{+}}{\implies} \nu_{\mathbf{R}^m_{+}, F}.
\]
Therefore, we \emph{sequentially} linearize the objective function from $f$ to its basis-linear relaxation $F$ and then relax the domain from $X$ to its cone hull $\text{cone}(X)\overset{\{\mathbf{0}_m, \mathbf{e}_1,\dots, \mathbf{e}_m\}\in X}{=}\mathbb{R}^m_{+}$. In contrast, Theorem~\ref{thm:nonlinear-packing-game}(i) skips the intermediate game $\nu_{X, F}$ by \emph{simultaneously} linearizing the objective function and relaxing the domain.

We establish the following equivalent version of Theorem~\ref{thm:nonlinear-packing-game}(i) by exploiting the intermediate game $\nu_{X, F}$, which will allow us to interpret Theorem~\ref{thm:nonlinear-packing-game}(i) from a new perspective and reveal deeper connections with previous linear objective games.
\begin{thm}\label{corr:packing-subadditive-intermediate-sandwich-equi} For the nonlinear game (\ref{game:packing-subadditive}) under assumptions (a-c), the core of the game $\nu_{X, f}$ is non-empty if and only if both of the following conditions hold:
\begin{enumerate}[(i)]
\item $\nu_{X, F}\left(\mathbf{1}_n\right)=\nu_{\mathbb{R}^m_{+}, F}\left(\mathbf{1}_n\right)$ as in (\ref{game:packing-subadditive-intermediate-X-F}) and (\ref{eq:NLP-relax-primal}), respectively, by choosing $w=\mathbf{1}_n$; and
\item there exists $x^*\in \argmax\limits_{x\in X} \left[F(x): Ax\le \mathbf{1}_n\right]$, as in (\ref{game:packing-subadditive-intermediate-X-F}) by choosing $w=\mathbf{1}_n$, such that $f(x^*)=F(x^*)$.
\end{enumerate}
\end{thm}
\begin{proof}
For necessity, if $\text{core}\left(\nu_{X, f}\right)\ne\emptyset$, then Theorem~\ref{thm:nonlinear-packing-game}(i) implies the first two inequalities in (\ref{ineq:nonlinear}) are actually equalities when choosing $w=\mathbf{1}_n$, and hence (i). Moreover, $\forall x^*\in \argmax\limits_{x\in X} \left[f(x): Ax\le \mathbf{1}_n\right]$ as in (\ref{game:packing-subadditive}) by choosing $w=\mathbf{1}_n$, we have $\nu_{X, F}\left(\mathbf{1}_n\right)=\nu_{X, f}\left(\mathbf{1}_n\right)=f(x^*)\le F(x^*)\le \nu_{X, F}\left(\mathbf{1}_n\right)$ implying equalities throughout, and hence (ii).

For sufficiency, $\nu_{X, F}\left(\mathbf{1}_n\right)\ge \nu_{X, f}\left(\mathbf{1}_n\right)\ge f(x^*)=F(x^*) =\nu_{X, F}\left(\mathbf{1}_n\right)=\nu_{\mathbb{R}^m_{+}, F}\left(\mathbf{1}_n\right)$ implying equalities throughout. Hence $\text{core}\left(\nu_{X, f}\right)\ne\emptyset$ due to Theorem~\ref{thm:nonlinear-packing-game}(i).
\end{proof}

With this new characterizations, the relationships among the games in this work, and the games studied in the previous literature become more transparent. 

For example, when $X=\mathbb{B}^m$ and $f(x)=c^Tx$ (and hence $f=F$), then 
\begin{eqnarray*}
\nu_{X, f}= \nu_{X, F}\le \nu_{\mathbf{R}^m_{+}, F}.
\end{eqnarray*}
Therefore, the intermediate game and the original game coincide, implying (ii) in Theorem~\ref{corr:packing-subadditive-intermediate-sandwich-equi} trivially holds and (i) in Theorem~\ref{corr:packing-subadditive-intermediate-sandwich-equi} is just \citep[Theorem 1]{deng1999algorithmic}. 

As another example, when $X=\mathbb{R}^m_{+}$ and $f(x)=c^Tx$, then 
\begin{eqnarray*}
\nu_{X, f}= \nu_{X, F}= \nu_{\mathbf{R}^m_{+}, F}.
\end{eqnarray*}
Therefore, all three games coincide, implying both (i) and (ii) in Theorem~\ref{corr:packing-subadditive-intermediate-sandwich-equi} trivially hold and Theorem~\ref{corr:packing-subadditive-intermediate-sandwich-equi} is just \citep[Theorem 4]{samet1984core}. 

Chronologically, for $X=\mathbb{B}^m$, our nonlinear game $\nu_{X, f}$ is an extension of the combinatorial game $\nu_{X, F}$~\citep{deng1999algorithmic}, which in turn is an extension of linear programming game $\nu_{\mathbf{R}^m_{+}, F}$~\citep{samet1984core}. However our main results not only extend the objective function from linear to nonlinear, but also extend the domain $X$ from the Boolean lattice $\mathbb{B}^m$ to more general constraint set, allowing us treating more games as we have demonstrated earlier and will be demonstrated further later on.

We illustrate Theorem~\ref{corr:packing-subadditive-intermediate-sandwich-equi} by using Example~\ref{exp:int-imlies-no-non-empty-core} again, where $X=\mathbb{B}^2, f=x_1+x_2-2x_1x_2, F=x_1+x_2$. The three games involved are as follows: 
\begin{eqnarray*}
\mathbb{B}^2\ni w\mapsto \nu_{\mathbb{B}^2,f}(w)&:=&\max_{(x_1, x_2)\in\mathbb{B}^2}\{x_1+x_2-2x_1x_2: x_1\le w_1, x_2\le w_2\};\\
\mathbb{B}^2\ni w\mapsto \nu_{\mathbb{B}^2,F}(w)&=&\max_{(x_1, x_2)\in\mathbb{B}^2}\{x_1+x_2: x_1\le w_1, x_2\le w_2\};\\
\mathbb{B}^2\ni w\mapsto \nu_{\mathbb{R}^2_{+},F}(w)&=&\max_{(x_1, x_2)\in\mathbb{R}^2_{+}}\{x_1+x_2: x_1\le w_1, x_2\le w_2\}.
\end{eqnarray*}
Then Theorem~\ref{corr:packing-subadditive-intermediate-sandwich-equi}(i) is satisfied because
$\nu_{\mathbb{B}^2, F}(1,1)=2=\nu_{\mathbb{R}^2_{+}, F}(1,1)$. However, Theorem~\ref{corr:packing-subadditive-intermediate-sandwich-equi}(ii) is violated because the only optimal solution of $\nu_{\mathbb{B}^2,F}(1,1)$ is $(x_1,x_2)=(1,1)$ whose optimal value $F(1,1)=2$ is not equal to $f(1,1)=0$. We reach the same conclusion that the core of this particular game $\nu_{\mathbb{B}^2,f}$ is empty.

\subsubsection{Characterization via the lower game}

Going downstream from the original game $\nu_{X,f}$ in (\ref{ineq:nonlinear}), we have the following schematic relationship, indicating how new games are constructed from the original game:
\[
\nu_{X, f}\overset{X\mapsto \overline{X}}{\implies} \nu_{\overline{X}, f}\overset{F|_{\overline{X}}=f|_{\overline{X}}}{\implies} \nu_{ \overline{X}, F}=\nu_{ \overline{X}, f}
\]

We obtain the following full characterization that is equivalent to Theorem \ref{thm:nonlinear-packing-game}.
\begin{thm} \label{thm:core_nonempty} For the nonlinear game (\ref{game:packing-subadditive}) under assumptions (a-c), the core of the game $\nu_{X, f}$ is non-empty if and only if $\nu_{\overline{X},F}\left(\mathbf{1}_n\right)=\nu_{\mathbb{R}^m_{+},F}\left(\mathbf{1}_n\right)$, as in (\ref{game:packing-subadditive-intermediate-X-cap-Y-F}) and (\ref{eq:NLP-relax-primal}), respectively, by choosing $w=\mathbf{1}_n$;.
\end{thm}
\begin{proof}
On the one hand, if the core of the game $\nu_{X,f}$ is non-empty, then $\nu_{\mathbb{R}^m_{+},F}\left(\mathbf{1}_n\right)=\nu_{X,F}\left(\mathbf{1}_n\right)=\nu_{X,f}\left(\mathbf{1}_n\right)$ hold by using Theorem \ref{thm:nonlinear-packing-game}. Let $x$ be an optimal solution of $\nu_{X,f}\left(\mathbf{1}_n\right)$.
Then $\nu_{X,F}\left(\mathbf{1}_n\right)=\nu_{X,f}\left(\mathbf{1}_n\right)=f(x)\leq F(x)\leq \nu_{X,F}\left(\mathbf{1}_n\right)$.
It follows that $F(x)=f(x)$, i.e., $x\in \overline{X}$.
Thus, we have $\nu_{X,f}\left(\mathbf{1}_n\right)\geq \nu_{\overline{X},f}\left(\mathbf{1}_n\right)\geq f(x)=\nu_{X,f}\left(\mathbf{1}_n\right)$.
This implies that $\nu_{X,f}\left(\mathbf{1}_n\right)=\nu_{\overline{X},f}\left(\mathbf{1}_n\right)=\nu_{\overline{X},F}\left(\mathbf{1}_n\right)$.

On the other hand, suppose that $\nu_{\mathbb{R}^m_{+},F}\left(\mathbf{1}_n\right)=\nu_{\overline{X},F}\left(\mathbf{1}_n\right)$, we obtain that the core is non-empty by using Theorem \ref{thm:nonlinear-packing-game} and inequalities~(\ref{ineq:nonlinear}).
\end{proof}

As two concrete applications of this new characterization, if $X=\mathbb{B}^m$ or $X=\mathbb{R}^m_{+}$ and $f(x)=c^Tx$ (and hence $f=F$), then $\overline{X}=X$ and
\begin{eqnarray*}
\nu_{\mathbf{R}^m_{+}, F}\ge \nu_{X, f}=\nu_{\overline{X}, f}= \nu_{\overline{X}, F}.
\end{eqnarray*}
Therefore, Theorem~\ref{thm:core_nonempty} is just \citep[Theorem 1]{deng1999algorithmic} when $X=\mathbb{B}^m$ or \citep[Theorem 4]{samet1984core} when $X=\mathbb{R}^m_{+}$.

We illustrate Theorem~\ref{thm:core_nonempty} by using Example~\ref{exp:int-imlies-no-non-empty-core} again, where $X=\mathbb{B}^2, f=x_1+x_2-2x_1x_2, F=x_1+x_2$. Then $\overline{X}=\{x\in \mathbb{B}^2: f(x)=F(x)\}=\{(0,0), (1,0), (0,1)\}$. The three games involved are as follows: 
\begin{eqnarray*}
\mathbb{B}^2\ni w\mapsto \nu_{X,f}(w)&:=&\max_{(x_1, x_2)\in\mathbb{B}^2}\{x_1+x_2-2x_1x_2: x_1\le w_1, x_2\le w_2\};\\
\mathbb{B}^2\ni w\mapsto \nu_{\overline{X},F}(w)&=&\max_{(x_1, x_2)\in\{(0,0), (1,0), (0,1)\}}\{x_1+x_2: x_1\le w_1, x_2\le w_2\};\\
= \nu_{\overline{X},f}(w)&=&\max_{(x_1, x_2)\in\{(0,0), (1,0), (0,1)\}}\{x_1+x_2-2x_1x_2: x_1\le w_1, x_2\le w_2\}\\
\mathbb{B}^2\ni w\mapsto \nu_{\mathbb{R}^2_{+},F}(w)&=&\max_{(x_1, x_2)\in\mathbb{R}^2_{+}}\{x_1+x_2: x_1\le w_1, x_2\le w_2\}.
\end{eqnarray*}
Then Theorem~\ref{thm:core_nonempty} is violated because $\nu_{\overline{X},F}(1,1)=\nu_{\overline{X},f}(1,1)=1\ne \nu_{\mathbf{R}^2_{+},F}(1,1)=2$. We reach the same conclusion that the core of this particular game $\nu_{\mathbb{B}^2,f}$ is empty.

\subsection{Two concrete examples}\label{subsec:examples-illustrative-concrete}

We present two concrete quadratic objective function examples to illustrate Theorem~\ref{thm:nonlinear-packing-game}, where the first has non-empty core and the second has an empty core. We will use the same constraint matrix as follows:
\[
A=\begin{pmatrix}
1&1&0&0\\
0&0&1&1\\
1&0&1&0\\
0&1&0&1
\end{pmatrix}
\]
\begin{example} Given a set function
\[
\mathbb{B}^4\ni x\mapsto f(x):=x_1+x_2+x_3+x_4-(x_1x_2+x_1x_3+x_2x_4+x_3x_4)
\]
its relaxation is
\[
\mathbb{R}^4\ni x\mapsto F(x):=x_1+x_2+x_3+x_4
\]
The nonlinear packing game is as follows:
\begin{eqnarray*}
\mathbb{B}^4\ni w\mapsto \nu_{\mathbb{B}^4,f}(w):=\max_{x\in \mathbb{B}^4} [f(x): Ax\le w]
\end{eqnarray*}
\end{example}
Note that the objective function $f$ is submodular and hence equivalent to individually subadditive (See Proposition~\ref{prop:id-sub}(v)). 

For this simple example, we can compute its characteristic function as follows:
\[
\nu_{\mathbb{B}^4,f}(w)=
\begin{cases}
0, &w=0, \mathbf{e}_i, i\in [4], \mathbf{e}_1+\mathbf{e}_2, \mathbf{e}_3+\mathbf{e}_4;\\
1, & w= \mathbf{e}_i+\mathbf{e}_j, i\in \{1,2\}, j\in \{3,4\}, \mathbf{e}_i+\mathbf{e}_j+\mathbf{e}_k, i\ne j\ne k\\
2,  & w=\mathbf{e}_1+\mathbf{e}_2+\mathbf{e}_3+\mathbf{e}_4=\mathbf{1}_4
\end{cases}
\]
Evidently the LP relaxation has an optimal value 2 and there exists an integer optimal solution $x_1=x_2=1, x_3=x_4=0$:
\begin{eqnarray*}
\nu_{\mathbb{R}^4,F}\left(\mathbf{1}_4\right)&:=&\max [F(x): Ax\le 1]\\
&=&\max[x_1+x_2+x_3+x_4: x_1+x_2\le 1, x_3+x_4\le 1, x_1+x_3\le 1,  x_2+x_4\le 1, x\ge 0]=2
\end{eqnarray*}
Therefore $\nu_{\mathbb{B}^4,f}\left(\mathbf{1}_4,f\right)=\nu_{\mathbb{R}^4,F}\left(\mathbf{1}_4\right)=2$. So Theorem~\ref{thm:nonlinear-packing-game} says the core of the game $\nu_{\mathbb{B}^4,f}$ is non-empty, and any dual optimal solution gives such a core member. The dual of the LP relaxation is as follow:
\[
\min[y_1+y_2+y_3+y_4: y_1+y_3\ge 1, y_1+y_4\ge 1, y_2+y_3\ge 1, y_2+y_4\ge 1, y\ge 0] 
\]
with optimal solution $y^*_1=y^*_2=y^*_3=y^*_4=\frac{1}{2}$, which can be verified to belong to the core. \qed

\begin{example} Given a slightly different set function
\[
\mathbb{B}^4\ni x\mapsto f(x):=(x_1+x_2+x_3+x_4)-(x_1x_2+x_1x_3+x_2x_4+x_3x_4+\underline{x_1x_4+x_2x_3})
\]
its relaxation is still the same
\[
\mathbb{R}^4\ni x\mapsto F(x):=x_1+x_2+x_3+x_4
\]
The nonlinear packing game is as follows:
\begin{eqnarray*}
\mathbb{B}^4\ni w\mapsto \nu_{\mathbb{B}^4,f}(w):=\max [f(x): Ax\le w]
\end{eqnarray*}
\end{example}

For this simple example, we can compute its characteristic function as follows:
\[
\nu_{\mathbb{B}^4,f}(w)=
\begin{cases}
0, &w=0, \mathbf{e}_i, i\in [4], \mathbf{e}_1+\mathbf{e}_2, \mathbf{e}_3+\mathbf{e}_4;\\
1,  & \text{otherwise}
\end{cases}
\]
Evidently the LP relaxation has an optimal value 2 and there exists an integer optimal solution $x_1=x_4=1, x_2=x_3=0$:
\begin{eqnarray*}
\nu_{\mathbb{R}^4,F}\left(\mathbf{1}_4\right)&:=&\max [F(x): Ax\le 1]\\
&=&\max[x_1+x_2+x_3+x_4: x_1+x_2\le 1, x_3+x_4\le 1, x_1+x_3\le 1,  x_2+x_4\le 1, x\ge 0]=2
\end{eqnarray*}
Therefore $\nu_{\mathbb{B}^4,f}\left(\mathbf{1}_4\right)=1\ne \nu_{\mathbb{R}^4,F}\left(\mathbf{1}_4\right)=2$. So Theorem~\ref{thm:nonlinear-packing-game} says the core of the game $\nu_{\mathbb{B}^4,f}$ is empty. Indeed if $(u_1,u_2,u_3,u_3)$ is a core member, then we have a contradiction:
\[
u_1+u_3\ge \nu_{\mathbb{B}^4,f}(\mathbf{e}_1+\mathbf{e}_3)=1, u_2+u_4\ge \nu_{\mathbb{B}^4,f}(\mathbf{e}_2+\mathbf{e}_4)=1\implies u_1+u_2+u_3+u_4\ge 2\ne 1=\nu_{\mathbb{B}^4,f}\left(\mathbf{1}_4\right),
\]
implying empty core. This is another example that shows that even the $\nu_{\mathbb{R}^4,F}\left(\mathbf{1}_n\right)$ has an integer optimal solution, the core is still empty. Moreover the empty core actually is due to the fact that this $\nu$ is not even grand superadditive. \qed

\section{Further results}\label{sec:further-results}

\subsection{Covering and partition games}\label{sub:cover-partition-games}
Analogous results for the nonlinear covering (cost) game and partition (revenue or cost) game are readily available with almost the same argument.
\begin{eqnarray*}
\mathbb{B}^n\ni w\mapsto \nu_{X, f}^{\text{cov}}(w)&:=&\min_{x\in X}  \left[f(x): Ax\ge w\right]\in  \mathbb{R}\\
\mathbb{B}^n\ni w\mapsto \nu_{X, f}^{\text{ptn}}(w)&:=&\max_{x\in X}  \left[f(x): Ax= w\right]\in  \mathbb{R}
\end{eqnarray*}
The corresponding relaxations along with their duals are
\begin{eqnarray*}
\mathbb{B}^n\ni w\mapsto \nu_{\mathbb{R}^m_{+}, F}^{\text{cov}}(w)
&=&\min_{x\in X} \left[(f(\mathbf{e}_1),\dots, f(\mathbf{e}_m)x: Ax\ge w\right]\\
&=&\max_{y\in \mathbb{R}^n_{+}} \left[w^Ty: y^TA\le (f(\mathbf{e}_1),\dots, f(\mathbf{e}_m)\right]\\ 
\mathbb{B}^n\ni w\mapsto  \nu_{\mathbb{R}^m_{+}, F}^{\text{ptn}}(w)&=&\max_{x\in X} \left[(f(\mathbf{e}_1),\dots, f(\mathbf{e}_m)x: Ax= w\right]\\
&=&\min_{y\in \mathbb{R}^n} \left[w^Ty: y^TA\le (f(\mathbf{e}_1),\dots, f(\mathbf{e}_m)\right] 
\end{eqnarray*}

\begin{thm}\label{thm:nonlinear-covering-game} (covering game)
Under assumptions (a-b) and $f(x): X\mapsto \mathbb{R}$ is individually supadditive, then 
\[
\text{core}\left(\nu_{X, f}^{\text{cov}}\right)\ne\emptyset \iff \nu_{X, f}^{\text{cov}}\left(\mathbf{1}_n\right)=\nu_{\mathbb{R}^m_{+}, F}^{\text{cov}}\left(\mathbf{1}_n\right).
\]
Moreover, whenever non-empty, 
\[
\text{core}\left(\nu_{X, f}^{\text{cov}}\right)= \argmax_{y\in \mathbb{R}^n_{+}} \left[\mathbf{1}_n^Ty: y^TA\le (f(\mathbf{e}_1),\dots, f(\mathbf{e}_m)\right].
\]
\end{thm}

\begin{thm}\label{thm:nonlinear-partition-game} (Partition game)
Under assumptions (a-b) and $f(x): X\mapsto \mathbb{R}$ is individually subadditive, then 
\[
\text{core}\left(\nu_{X, f}^{\text{ptn}}\right)\ne\emptyset \iff \nu_{X, f}^{\text{cov}}\left(\mathbf{1}_n\right)=\nu_{\mathbb{R}^m_{+}, F}^{\text{ptn}}\left(\mathbf{1}_n\right).
\]
Moreover, whenever non-empty, 
\[
\text{core}\left(\nu_{X, f}^{\text{cov}}\right)= \argmin_{y\in \mathbb{R}^n_{+}} \left[\mathbf{1}_n^Ty: y^TA\le (f(\mathbf{e}_1),\dots, f(\mathbf{e}_m)\right].
\]
\end{thm}

\subsection{$X$ with a finite number of generators}\label{subsec:gene-X-constraint-pack-game}

In the nonlinear game (\ref{game:packing-subadditive}), there are three assumptions in Theorem~\ref{thm:nonlinear-packing-game}: (i) the domain $X\subseteq \mathbb{R}^m_{+}$ includes $\{\mathbf{0}_m, \mathbf{e}_1,\dots, \mathbf{e}_m\}$; (ii) the objective function $f$ is individually subadditive; namely $f$ is dominated by its basis-linear relaxation $F$; and (iii) the constraint matrix $A\in \mathbb{B}^{n\times m}$ has zero-one entries without any zero column; or equivalently $A\mathbf{e}_j\in \mathbb{B}^n\backslash \{\mathbf{0}_n\}, \forall j\in [m]$. A closer look at the proof reveals that these assumptions can be extended to more general settings. 

If we replace these three assumptions as follows, then the proof of Theorem~\ref{thm:nonlinear-packing-game} will carry through analogously by substituting LP duality with conic LP duality. 
\begin{enumerate}[(i)]
\item The sole purpose of including $\mathbf{0}_m$ in $X$ is guarantee the groundness of the game $\nu_{X, f}(\mathbf{0}_n)=0$ (plus assumption $f(\mathbf{0}_m)=0)$. On the other hand, the set $\{\mathbf{e}_1,\dots, \mathbf{e}_m\}\subseteq X$ can be replaced by any set of independent generators $Q=\{q_1,\dots, q_k\}\subseteq X$ satisfying $\text{cone}(X)=\text{cone}(Q)$, where $k\le m$. 

Hence $Q=\{\mathbf{e}_1,\dots, \mathbf{e}_m\}$ is just a special set of independent generators with $\text{cone}(X)=\text{cone}(Q)=\text{cone}\left(\mathbf{e}_1,\dots, \mathbf{e}_m\right)=\mathbb{R}^m_{+}$, which is the domain we used earlier in the relaxed anchor game $\nu_{\mathbb{R}^m_{+}, F}$. Now the relaxed anchor game will be the more general $\nu_{\text{cone}(X), F}=\nu_{\text{cone}(Q), F}$ instead.

Note also independence in $Q$ and $\text{cone}(X)=\text{cone}(Q)$ together imply that any $x\in X\subseteq \mathbf{R}^m_{+}$ can be represented uniquely as the conic combination of the generators in $Q$ via the pseudo-inverse. Here is the construction. 

Let $Q\in\mathbb{R}^{m\times k} (k\le m)$ also represents the matrix whose independent columns are $q_1,\dots, q_k$. Then their independence means $Q$ has full column rank $k$. Hence $Q$ has the left Moore–Penrose inverse $Q^{\dagger}=(Q^TQ)^{-1}Q^T\in\mathbb{R}^{k\times m}$ such that $Q^{\dagger}Q=I_k$. We have the following unique representation for any $x\in X\subseteq \text{cone}(X)=\text{cone}(Q)\subseteq \mathbb{R}^m_{+}$ by using columns of $Q$ as the basis instead of the unit basis:
\begin{equation}\label{eq:x-rep-Q}
x=QQ^{\dagger}x, \text{ where } Q^{\dagger}x\ge 0.
\end{equation}
\item Individually subadditiveness is replaced by $Q$-\emph{individually subadditive} as follows:
 \begin{deff}\label{deff:indiv-supadd_Q} ($Q$-\textbf{individually subadditive}) Given a subset $X\subseteq \mathbb{R}^m_{+}$ such that $\text{cone}(X)=\text{cone}(Q)$ for some independent $Q=\{q_1,\dots, q_k\}\subseteq X$, a function $f: X\mapsto \mathbb{R}$ is $Q$-\emph{individually subadditive} if it satisfies:
\begin{equation*}
f(x)=f\left(QQ^{\dagger}x\right)\le \left(f(q_1),\dots, f(q_k)\right)Q^{\dagger}x
\end{equation*}
\end{deff}
Hence when $Q=\{\mathbf{e}_1,\dots, \mathbf{e}_m\}$, the matrix $Q=I_{m\times m}$ and it reduces to individually subadditiveness. 
\item $A\mathbf{e}_j\in \mathbb{B}^n\backslash \{\mathbf{0}_n\}, \forall j\in [m]$ is replaced with $Aq_j\in \mathbb{B}^n\backslash \{\mathbf{0}_n\}, \forall j\in [k]$.
\end{enumerate}

Similarly as before, we call the function on the right-hand-side of Definition~\ref{deff:indiv-supadd_Z^n} the $Q$-\emph{basis-linear} relaxation, which is well-defined for any function $f: X\mapsto \mathbb{R}$, not just subadditive functions.
\begin{deff}\label{deff:Q-basis-linear-relaxation} ($Q$-\textbf{basis-linear relaxation}) Given a function $X\ni x\mapsto f(x)\in\mathbb{R}$, its $Q$-basis-linear relaxation is given as follows: 
\[
\text{cone}(X)\ni x\mapsto F(x):=\left(f(q_1),\dots, f(q_k)\right)Q^{\dagger}x\in\mathbb{R}
\]
\end{deff}

We consider the following combinatorial packing game:
\begin{eqnarray}\label{game:packing-subadditive-const-Q-basis}
 \mathbb{B}^n\ni w\mapsto \nu_{X, f}(w)&:=&\max\limits_{x\in X}\left[f(x): Ax\le w\right]\in \mathbb{R},
\end{eqnarray}
with the same setting as in the game (\ref{game:packing-subadditive}) except for the constraint $X\subseteq \mathbb{B}^n$, there exists independent $Q=\{q_1,\dots, q_k\}\subseteq X$ such that $\text{cone}(X)=\text{cone}(Q)$, and the $Q$-individually subadditive function $f(x): \mathbb{B}^n\mapsto\mathbb{R}$ with $f(\mathbf{0}_m)=0$. Therefore we go back to the original setting when $q_i=\mathbf{e}_i, i\in [k]$.

The relaxed anchor game along with its conic LP dual are as follows via the $Q$-basis-linear relaxation as in Definition~\ref{deff:Q-basis-linear-relaxation}:
\begin{scriptsize}
\begin{eqnarray}
\label{eq:NLP-relax-primal-Q-basis}\mathbb{B}^n\ni w\mapsto \nu_{\text{cone}(X), F}\left(w\right)&=&\max\limits_{x\in \text{cone}(X)} \left[F(x): Ax\le w\right]\\
\nonumber&=&\max\limits_{x\in \mathbb{R}^m_{+}} \left[F(x): w-Ax\in\mathbb{R}^n_{+}, x\in \text{cone}(X)\right]\\
\nonumber&\overset{\text{Conic LP dual}}{=}&\min_{y\in\mathbb{R}^n}\left[w^Ty: -(Q^{\dagger})^T\left(f(q_1),\dots, f(q_k)\right)^T+A^Ty\in \left(\text{cone}(X)\right)^*, y\ge \left(\mathbb{R}^n_{+}\right)^*\right]\\
\label{eq:NLP-relax-dual-Q-basis} &=&\min_{y\in\mathbb{R}^n_{+}}\left[w^Ty: y^TAQ\ge \left(f(q_1),\dots, f(q_k)\right)\right],
\end{eqnarray}
where the last equality used the following facts:
\begin{eqnarray*}
\left(\text{cone}(X)\right)^*&=&\left(\text{cone}(Q)\right)^*=\left\{y\in\mathbb{R}^m: Q^Ty\ge 0\right\} \\
\left(\mathbb{R}^n_{+}\right)^*&=&\mathbb{R}^n_{+}
\end{eqnarray*}
\end{scriptsize}

Then the following result can be shown analogously, whose proof is in Appendix A. 
\begin{thm}\label{thm:nonlinear-packing-const-X-generators}
In game (\ref{game:packing-subadditive-const-Q-basis}), assume that (a) the function $f(x): \mathbb{B}^n\mapsto\mathbb{R}$ with $f(\mathbf{0}_m)=0$ is $Q$ individually subadditive as in Definition~\ref{deff:indiv-supadd_Z^n} and $F$ is the $Q$-basis linear relaxation of $f$ as in Definition~\ref{deff:Q-basis-linear-relaxation}; (b) there exists a set of independent generators $Q=\{q_1,\dots, q_k\}\subseteq X$ such that $\text{cone}(X)=\text{cone}(Q)$; and (c)  $Aq_j\in \mathbb{B}^n\backslash \{\mathbf{0}_n\}, \forall j\in [k]$. Then 
\begin{enumerate}[(i)]
\item The core of the game $\nu_{X, f}$ is non-empty if and only if $\nu_{X, f}\left(\mathbf{1}_n\right)=\nu_{\text{cone}(X), F}\left(\mathbf{1}_n\right)$ as in (\ref{game:packing-subadditive-const-Q-basis}) and (\ref{eq:NLP-relax-primal-Q-basis}), respectively, by choosing $w=\mathbf{1}_n$.
\item The core of $\nu_{X, f}$, whenever non-empty, coincides with the set of optimal solutions of the dual LP as in (\ref{eq:NLP-relax-dual-Q-basis}) by choosing $w=\mathbf{1}_n$; namely
 \[
 \text{core} \left(\nu_{X, f}\right)=\argmin\limits_{y\in \mathbb{R}^n_{+}} \left[\mathbf{1}_n^Ty: y^TAQ\ge \left(f(q_1),\dots, f(q_k)\right)\right]
 \]
\end{enumerate}
\end{thm}

Theorem~\ref{thm:nonlinear-packing-const-X-generators} includes \citep[Theorem 2.1]{goemans2004cooperative} for the facility location games as a special case, where only linear objective function and special constraint matrix $A$ are considered.

\subsection{$X$ dependent on $w$}\label{subsec:gene-constraint-pack-game}
We consider the following combinatorial packing game:
\begin{eqnarray}\label{game:packing-subadditive-const-dep-w}
 \mathbb{B}^n\ni w\mapsto \nu_{X, f}(w)&:=&\max\limits_{x\in X(w)}\left[f(x): Ax\le w\right]\in \mathbb{R},
\end{eqnarray}
with the same setting as in the game (\ref{game:packing-subadditive}) except the constraint $X(w)\subseteq \mathbb{B}^n$ also depends on the players such that $\mathbf{0}_m\in X(0), \mathbf{e}^m_i\in X(\mathbf{e}^n_i), i\in [n]$ and the function $f(x): \mathbb{B}^n\mapsto\mathbb{R}$ with $f(\mathbf{0}_m)=0$.
Then the LP relaxed game along with its LP dual as follows via the basis-linear relaxation as in Definition~\ref{deff:basis-linear-relaxation}:
\begin{eqnarray}
\label{eq:NLP-relax-primal-const-dep-w}\mathbb{B}^n\ni w\mapsto \nu_{\mathbb{R}^m_{+}, F}\left(w\right)&=&\max\limits_{x\in \mathbb{R}^m_{+}} \left[F(x): Ax\le w\right]\\
\label{eq:NLP-relax-dual-const-dep-w}&\overset{\text{LP dual}}{=}&\min_{y\in\mathbb{R}^n_{+}}\left[w^Ty: y^TA\ge (f(\mathbf{e}_1),\dots, f(\mathbf{e}_m))\right]
\end{eqnarray}

Then the following result can be shown analogously, whose proof is in Appendix B. 
\begin{thm}\label{thm:nonlinear-packing-const-dep-w}
In game (\ref{game:packing-subadditive-const-dep-w}), assume that the function $f(x): \mathbb{B}^n\mapsto\mathbb{R}$ with $f(\mathbf{0}_m)=0$  is individually subadditive as in Definition~\ref{deff:indiv-supadd_Z^n}, and $X(w)\subseteq \mathbb{B}^n$ satisfies that $\mathbf{0}_m\in X(w), \forall w\in \mathbb{B}^n,  \mathbf{e}^m_i\in X(\mathbf{e}^n_i), i\in [n]$.  
\begin{enumerate}[(i)]
\item The core of the game $\nu_{X, f}$ is non-empty if and only if $\nu_{X, f}\left(\mathbf{1}_n\right)=\nu_{\mathbb{R}^m_{+}, F}\left(\mathbf{1}_n\right)$ as in (\ref{eq:NLP-relax-primal-const-dep-w}) and (\ref{eq:NLP-relax-dual-const-dep-w}), respectively, by choosing $w=\mathbf{1}_n$.
\item The core of $\nu_{X, f}$, whenever non-empty, coincides with the set of optimal solutions of the dual LP as in (\ref{eq:NLP-relax-dual-const-dep-w}) by choosing $w=\mathbf{1}_n$; namely
 \[
 \text{core} \left(\nu_{X, f}\right)=\argmin\limits_{y\in \mathbb{R}^n_{+}} \left[\mathbf{1}_n^Ty: y^TA\ge (f(\mathbf{e}_1),\dots, f(\mathbf{e}_m))\right]
 \]
\end{enumerate}
\end{thm}

\subsection{More general right-hand side}\label{subsec:gene-nl-pack-game}

The nonlinear games requires the right-hand side of the constraint to be a binary vector, which prevents us from attacking the capacitated version of combinatorial games, such as the maximum flow game with common capacity $b\in\mathbb{R}_{++}$ and the bipartite $b$-matching problem, treated in \citep{vazirani2023lp} who show that the cores of these games can also be fully characterized. 

The rest of this section will be devoted to generalize these particular games. 

We consider the following nonlinear game with a slightly more general right-hand side, which may take any positive value $b\in \mathbb{R}_{++}$ and zero rather than just one and zero. The following game is the same as game (\ref{game:packing-subadditive}) expect we replace the right hand side by $bw$. 
\begin{eqnarray}
\label{game:packing-subadditive-rhs}\mathbb{B}^n\ni w\mapsto \nu_{X, f}(w):=\max\limits_{x\in X} \left[f(x): Ax\le bw\right]\in  \mathbb{R},
\end{eqnarray}
with the same setting as in the game (\ref{game:packing-subadditive}) except the domain $X$ satisfies $\{\mathbf{0}_n, b\mathbf{e}_1,\dots, b\mathbf{e}_m\}\subseteq X\subseteq \mathbb{R}^m$, and the right-hand side of the constraint is $bw$ instead of $w$.

We extend the concept of individually subadditive to $b$-\emph{individually subadditive}.
\begin{deff}\label{deff:indiv-supadd_Z^n} Given a positive number $b\in\mathbb{R}_{++}$ and a subset $\{b\mathbf{e}_1,\dots, b\mathbf{e}_m\}\subseteq X\subseteq \mathbb{R}^m$, a function $f: X\mapsto \mathbb{R}$ is $b$-\emph{individually subadditive} if it satisfies:
\begin{equation*}
f(x)\le b^{-1}\left(x_1f(b\mathbf{e}_1)+\dots+x_mf(b\mathbf{e}_m)\right), \forall x\in X
\end{equation*}
\end{deff}
A function $f$ is $b$-\emph{individually superadditive} if and only if $-f$ is $b$-\emph{individually subadditive}. Evidently linear functions are both $b$-\emph{individually subadditive} and $b$-\emph{individually superadditive}, and hence can also be equivalently called $b$-\emph{individually additive} or just additive.

Similarly as before, we call the function on the right-hand-side of Definition~\ref{deff:indiv-supadd_Z^n} the $b$-\emph{basis-linear} relaxation, which is well-defined for any function $f: X\mapsto \mathbb{R}^m$, not just  subadditive functions.
\begin{deff}\label{deff:basis-linear-relaxation} ($b$-\textbf{basis-linear relaxation}) Given a function $\mathbb{R}^m\ni x\mapsto f(x)\in\mathbb{R}$, its $b$-basis-linear relaxation is given as follows: $\forall b\in\mathbb{R}_{++}$, 
\[
\mathbb{R}^m\ni x\mapsto F(x):=b^{-1}(f(b\mathbf{e}_1),\dots, f(b\mathbf{e}_m))x\in\mathbb{R}
\]
\end{deff}

Then the LP relaxed game along with its LP dual as follows via the $b$-basis-linear relaxation as in Definition~\ref{deff:basis-linear-relaxation}:
\begin{eqnarray}
\label{eq:NLP-relax-primal-rhs}\mathbb{B}^n\ni w\mapsto \nu_{\mathbb{R}^m_{+}, F}\left(w\right)&=&\max\limits_{x\in \mathbb{R}^m_{+}} \left[F(x): Ax\le bw\right]\\
\nonumber&\overset{\text{LP dual}}{=}&\min_{z\in\mathbb{R}^n_{+}}\left[bw^Tz: z^TA\ge b^{-1}(f(b\mathbf{e}_1),\dots, f(b\mathbf{e}_m))\right] \\
\label{eq:NLP-relax-dual-rhs}&\overset{y:=bz}{=}&\min_{y\in\mathbb{R}^n_{+}}\left[w^Ty: y^TA\ge (f(b\mathbf{e}_1),\dots, f(b\mathbf{e}_m))\right]
\end{eqnarray}

Note that we choose to define the dual program DLP($w$) after variable transform in the $y$-variables, which will directly correspond to core members as shown shortly. This may differ from the treatment in the existing literature where the dual program is referred to the one before the variable transform in the $z$-variables, and core members correspond to $bz$.

Then the following result can be shown analogously, whose proof is in Appendix C. 
\begin{thm}\label{thm:nonlinear-packing-rhs-arb}
In game (\ref{game:packing-subadditive-rhs}), assume that $f(x)$ is $b$-individually subadditive as in Definition~\ref{deff:indiv-supadd_Z^n}. 
\begin{enumerate}[(i)]
\item The core of the game $\nu_{X, f}$ is non-empty if and only if $\nu_{X, f}\left(\mathbf{1}_n\right)=\nu_{\mathbb{R}^m_{+}, F}\left(\mathbf{1}_n\right)$ as in (\ref{game:packing-subadditive-rhs}) and (\ref{eq:NLP-relax-primal-rhs}), respectively, by choosing $w=\mathbf{1}_n$.
\item The core of $\nu_{X, f}$, whenever non-empty, coincides with the set of optimal solutions of the dual LP as in (\ref{eq:NLP-relax-dual-rhs}) by choosing $w=\mathbf{1}_n$; namely
 \[
 \text{core} \left(\nu_{X, f}\right)=\argmin\limits_{y\in \mathbb{R}^n_{+}} \left[\mathbf{1}_n^Ty: y^TA\ge (f(\mathbf{e}_1),\dots, f(\mathbf{e}_m))\right]
 \]
\item For nonlinear combinatorial game where $X=\mathbb{B}^m$, whenever the core is non-empty, the LP, as in (\ref{eq:NLP-relax-primal-rhs}) by choosing $w=\mathbf{1}_n$, has an integer optimal solution. 
\end{enumerate}
\end{thm}

Because linear function is $b$-individually subadditive (actually $b$-additive), Theorem~\ref{thm:nonlinear-packing-rhs-arb} can be applied to imply immediately the full characterization of the core for the uniform bipartite $b$-matching game considered in \citep[Theorem 9]{vazirani2023lp}, which includes the bipartite $1$-matching game \citep[Theorem 3]{deng1999algorithmic} as a special case. 

Moreover, as a bonus, if we specialize the maximum flow game in \citep[Section 8]{vazirani2023lp} to the $b$-uniform capacity case, then Theorem~\ref{thm:nonlinear-packing-rhs-arb} can again be applied to offer a full characterization of the core for the $b$-uniform capcaitated maximum flow game, which includes the $1$-uniform capcaitated maximum flow game \citep[Section 3.1]{deng1999algorithmic} as a special case. .

\subsection{More general objective function}\label{subsec:general-obj}
We consider the following packing game:
\begin{eqnarray}\label{game:packing-subadditive-obj}
 \mathbb{B}^n\ni w\mapsto \nu_{X, f}(w)&:=&\max\limits_{x\in X}\left[f(x, w): Ax\le w\right]\in \mathbb{R},
\end{eqnarray}
with the same setting as in the game (\ref{game:packing-subadditive}) except the function $f(x, w): X\times \mathbb{B}^n\mapsto\mathbb{R}$ with $f(\mathbf{0}_m, \mathbf{0}_n)=0$ also depends on the players. 

We extend the concept of individually subadditive to $A$-\emph{individually subadditive}.
\begin{deff}\label{def:matrix-subadditive} (\textbf{$A$-individually subadditive}) Given $X\subseteq \mathbb{R}^m$, a function $f(x, w): X\times \mathbb{B}^n\mapsto \mathbb{R}$ with $f(\mathbf{0}_m, \mathbf{0}_n)=0$ is $A$-individually subadditive for a given matrix $A\in\mathbb{B}^{n\times m}$, if 
\[
f(x,w)\le \left(f\left(\mathbf{e}_1, A\mathbf{e}_1\right),\dots, f\left(\mathbf{e}_m, A\mathbf{e}_m\right)\right)x, \forall (x,w)\in X\times \mathbb{B}^n.
\]
\end{deff}

With this concept, we call the function on the right-hand-side of Definition~\ref{def:matrix-subadditive} the $A$-\emph{basis-linear} relaxation, which is well-defined for any function $f: X\times \mathbb{B}^m \mapsto \mathbb{R}$, not just $A$-individually subadditive functions.
\begin{deff}\label{deff:basis-linear-relaxation} ($A$-\textbf{basis-linear relaxation}) Given $X\subseteq \mathbb{R}^m$, and a function $f: X\times \mathbb{B}^n\mapsto \mathbb{R}$, its basis-linear relaxation $F$ is given as follows:
\[
X\times \mathbb{B}^m\ni (x, w)\mapsto F(x,w):=\left(f\left(\mathbf{e}_1, A\mathbf{e}_1\right),\dots, f\left(\mathbf{e}_m, A\mathbf{e}_m\right)\right)x\in\mathbb{R}
\]
\end{deff}

With the $A$-basis-linear relaxation, we introduce the following relaxed game along with its dual:
\begin{eqnarray}
\label{eq:NLP-relax-primal-obj}
\mathbb{B}^n\ni w\mapsto \nu_{\mathbb{R}^m_{+}, F}(w)&:=&\max\limits_{x\in \mathbb{R}^m_{+}} \left[F(x,w): Ax\le w\right] \\
\label{eq:NLP-relax-dual-obj}&\overset{\text{LP duality}}{=}&\min\limits_{y\in \mathbb{R}^n_{+}} \left[w^Ty: y^TA\ge \left(f\left(\mathbf{e}_1, A\mathbf{e}_1\right),\dots, f\left(\mathbf{e}_m, A\mathbf{e}_m\right)\right)\right]
\end{eqnarray}

Now we have the following result, whose proof is in Appendix D. 
\begin{thm}\label{thm:nonlinear-packing-game-obj-constraint}
In game (\ref{game:packing-subadditive-obj}), assume that $f(x, w)$ is $A$-individually subadditive as in Definition~\ref{def:matrix-subadditive} and monotonically increasing in $w$. 
\begin{enumerate}[(i)]
\item The core of the game $\nu_{X, f}$ is non-empty if and only if $\nu_{X, f}\left(\mathbf{1}_n\right)=\nu_{\mathbb{R}^m_{+}, F}\left(\mathbf{1}_n\right)$ as in (\ref{game:packing-subadditive-obj}) and (\ref{eq:NLP-relax-primal-obj}), respectively, by choosing $w=\mathbf{1}_n$.

\item The core of $\nu_{X, f}$, whenever non-empty, coincides with the set of optimal solutions of the dual LP as in (\ref{eq:NLP-relax-dual-obj}) by choosing $w=\mathbf{1}_n$; namely
 \[
 \text{core} \left(\nu_{X, f}\right)=\argmin\limits_{y\in \mathbb{R}^n_{+}} \left[\mathbf{1}_n^Ty: y^TA\ge (f(\mathbf{e}_1),\dots, f(\mathbf{e}_m))\right]
 \]
\item For nonlinear combinatorial game where $X=\mathbb{B}^m$, whenever the core is non-empty, the LP, as in (\ref{eq:NLP-relax-primal-obj}) by choosing $w=\mathbf{1}_n$, has an integer optimal solution. 
\end{enumerate}
\end{thm}

\subsection{Approximate core}\label{subsec:appro-core}
The full characterization for core can be easily extended to that for approximation core defined as follows. 
\begin{deff}\label{def:appro-core} (\textbf{$\gamma$-approximate core})~\citep[Def. 15.7; Page 389]{nisan2007algorithmic} and \citep[Lemma 2.2; Page 20]{qui2013fractional}. For $\gamma\ge 1$, 
\[
\gamma\text{-core}(\nu)=\{y\in\mathbb{R}^n: \mathbf{1}_n^Ty\le \gamma \nu\left(\mathbf{1}_n\right), a^Ty\ge \nu(a), \forall a\in\mathbb{B}^n\}
\]
\end{deff}

Then we have the following result, whose proof is in Appendix E. 
\begin{thm}\label{thm:appro-core} For the nonlinear game (\ref{game:packing-subadditive}) under assumptions (a-c), we have
\begin{enumerate}[(i)]
\item 
\[
\gamma\text{-core}\left(\nu_{X, f}\right)\ne\emptyset \iff \frac{\nu_{\mathbb{R}^m_{+}, F}\left(\mathbf{1}_n\right)}{\nu_{X, f}\left(\mathbf{1}_n\right)}\le \gamma.
\]
\item Whenever non-empty
\[
\gamma\text{-core}\left(\nu_{X, f}\right)=\argmin\limits_{y\in\mathbb{R}^{n}_{+}}\left\{\mathbf{1}_n^Ty: y^TA\ge (f(\mathbf{e}_1),\dots, f(\mathbf{e}_m))x\right\}
\]
\end{enumerate}
\end{thm}

\section{Applications}\label{sec:apps}

We demonstrate the utilities of Theorems~\ref{thm:nonlinear-packing-game}-\ref{thm:nonlinear-partition-game} by applying them to various games in this section. We focus on quadratic and linear fractional objective functions with or without constraints which are important classes of nonlinear functions and already have many applications.

\subsection{Combinatorial quadratic games}
We consider combinatorial quadratic games in this section. 
\begin{eqnarray}
\label{game:packing-subadditive-quadratic}\mathbb{B}^n\ni w\mapsto \nu_{X, f}(w):=\max\limits_{x\in \mathbb{B}^m} \left[b^Tx+x^TQx: Ax\le w\right]\in  \mathbb{R},
\end{eqnarray}
under the assumptions (a-c), where $b\in \mathbb{R}^m, Q\in \mathbb{R}^{m\times m}$ with $Q^T=Q$, and $X=\mathbb{B}^m$. From Proposition~\ref{prop:id-sub}(v), the individual subadditiveness of the quadratic objective function $f(x)=b^Tx+x^TQx$ required in assumption (c) is equivalent to submodularity of $f$, which in turn is equivalent to non-positive off-diagonal entries in $Q$ for quadratic function.

The basis-linear relaxation of $f$ is $F(x)=b^Tx+\sum_{i\in [m]}q_{ii}x_i=(b+q)^Tx$, where $q^T=(q_{11},\dots, q_{nn})$. Moreover, we have
\begin{eqnarray*}
X=\mathbb{B}^m\implies \overline{X}=\overline{\mathbb{B}}^n&=&\left\{x\in \mathbb{B}^m: f(x)=F(x)\right\}=\left\{x\in\mathbb{B}^m: \sum_{i\ne j\in [m]}q_{ij}x_ix_j=0\right\}\\
&=&\left\{x\in \mathbb{B}^m: q_{ij}x_ix_j=0, \forall i\ne j\in [m]\right\}\\
&=&\left\{x\in \mathbb{B}^m: q_{ij}=0, \forall i\ne j\in \text{supp}(x)\right\},
\end{eqnarray*}
where $\text{supp}(x):=\left\{i\in [m]: x_i=1\right\}$. So the relaxed anchor and lower games are as follows: 
\begin{eqnarray*}
\mathbb{B}^n\ni w\mapsto \nu_{\mathbb{R}^m_{+},F}(w)&:=&\max\limits_{x\in \mathbb{R}^m_{+}} \left[(b+q)^Tx: Ax\le w\right]\\
\mathbb{B}^n\ni w\mapsto \nu_{\overline{\mathbb{B}}^n, F}(w)&:=&\max\limits_{x\in \overline{\mathbb{B}}^n} \left[(b+q)^Tx: Ax\le w\right]\\
&=&\max\limits_{x\in \mathbb{B}^m} \left[(b+q)^Tx: Ax\le w, q_{ij}x_ix_j=0, \forall i\ne j\in [m]\right]
\end{eqnarray*}

Theorem \ref{thm:core_nonempty} implies that the core of the game $\nu_{\mathbb{B}^m,f}$ is non-empty if and only if $\nu_{\overline{\mathbb{B}}^n, F}\left(\mathbf{1}_n\right)=\nu_{\mathbb{R}^m_{+},F}\left(\mathbf{1}_n\right)$. 

This characterization becomes more explicit when we consider special constraint matrix $A$, as will be done below. In particular we consider two classes of applications. The first is when $A$ is the identity matrix $I_{n}$, which models games like the portfolio game and the maximum cut game. The second is when $A$ is the incidence matrix of a graph, which models games like the matching game.

\subsubsection{Portfolio games and beyond}\label{sec:unconstrained-qp}
We assume that $A=I_{n}$ throughout this section, and hence $m=n$. 
Then 
\begin{eqnarray*}
\nu_{\mathbb{R}^n_{+},F}\left(\mathbf{1}_n\right)&=&\max\limits_{x\in \mathbb{R}^n_{+}} \left[(b+q)^Tx: x\le \mathbf{1}_n\right]=\mathbf{1}_n^T(b+q)^{+}\\
\nu_{\overline{\mathbb{B}}^n, F}\left(\mathbf{1}_n\right)&=&\max\limits_{x\in \mathbb{B}^n} \left[(b+q)^Tx: q_{ij}x_ix_j=0, \forall i\ne j\in [n]\right],
\end{eqnarray*}
We have the following full characterization for core non-emptiness. 
\begin{prop}\label{prop:quad-game} The combinatorial quadratic game (\ref{game:packing-subadditive-quadratic}) when $A=I_n$ has a non-empty core if and only if for all distinct pair 
$\forall i, j\in \left\{k\in [n]: (b+q)^T\mathbf{e}_k>0\right\}$, we have $q_{ij}=0$. Moreover, whenever nonempty, the core has a unique member $y=(b+q)^{+}$.
\end{prop}
\begin{proof} Suffice being evident, for necessity, assume that the core is non-empty, namely $\nu_{\overline{\mathbb{B}}^n, F}\left(\mathbf{1}_n\right)=\mathbf{1}_n^T(b+q)^{+}$. Suppose on the contrary, there exists a distinct pair $\forall i, j\in \left\{k\in [n]: (b+q)^T\mathbf{e}_k>0\right\}$ such that we have $q_{ij}<0$. Then either $x_i=0$ or $x_j=0$, implying either $(b+q)^T\mathbf{e}_i$ or $(b+q)^T\mathbf{e}_j$ will be missing from $\mathbf{1}_n^T(b+q)^{+}$, a contradiction.

The dual program is as follows:
\begin{eqnarray*}
\nu_{\mathbb{R}^n_{+},F}\left(\mathbf{1}_n\right)&=&\max\limits_{x\in \mathbb{R}^n_{+}} \left[(b+q)^Tx: x\le \mathbf{1}_n\right]=\min\limits_{y\in \mathbb{R}^n_{+}} \left[\mathbf{1}_n^Ty: y\ge b+q\right]=\mathbf{1}_n^T(b+q)^{+}
\end{eqnarray*}
Then $y=(b+q)^{+}$ is a dual optimal solution, which is the unique member of the core, whenever nonempty. 

\end{proof}

The last result implies that testing core non-emptiness for this game can be done in polynomial time. We can also offer an equivalent statement in graph terms. Construct a vertex-weighted and edge-weighted graph $(V=[n], E=[n]\times [n])$ where each vertex $i\in [n]$ has a weight weight $(b+q)^T\mathbf{e}_i$, and each edge $(i,j)\in E$ has a weight $q_{ij}$. Then the core of the game $\nu$ is non-empty if and only if an edge with a negative weight exists if and only if at least one of its adjacent vertices has a non-positive weight.

As practical applications of the above combinatorial quadratic model, we present two concrete examples.  
\begin{enumerate}[(i)]
\item Discrete portfolio game arising from discrete portfolio problem~\citep{olszewski2014selecting} (a.k.a., project or resource allocation problem~\citep{IbarakiKatoh1988}): 
\[
\mathbb{B}^n\ni w\mapsto \nu(w):=\max_{x\in \mathbb{B}^n}\left[f(x):=\mu^Tx-\frac{\gamma}{2}x^T\Sigma x: x\le w\right]\in\mathbb{R},
\]
where $\mu\in\mathbb{R}^n, \Sigma=(\sigma_{ij})_{i,j\in [n]}\in\mathbb{R}^{n\times n}$ are the mean return vector and covariance matrix of $n$ assets, and $\gamma>0$ is a risk-averse factor. For each asset $i\in [n]$ as a player, $x_i$ is a binary decision variable indicating whether asset $i$ is selected or not. For positively correlated assets or projects, we have $\Sigma\ge 0$, implying submodular (individually subadditive) objective function $f$. 

Let $q:=(\sigma_{11},\dots, \sigma_{nn})^T$ be the variance vector of the $n$ assets. Then Proposition~\ref{prop:quad-game} implies that $\nu$ has a non-empty core if and only if for all distinct pair
$\forall i, j\in \left\{k\in [n]: \left(\mu-\frac{\gamma}{2}q\right)^T\mathbf{e}_k>0\right\}$, we have $\sigma_{ij}=0$. Economically, this means that core allocation is possible only when any two assets with individual positive risk-adjusted returns must be the uncorrelated. Moreover, whenever nonempty, the core has a unique member $y=\left(\mu-\frac{\gamma}{2}q\right)^{+}$.

\item Maximum cut game arising from the maximum cut problem~\citep{goemans1995improved}: 
\begin{eqnarray*}
\nu(w)&:=&\max\limits_{x\in\mathbb{B}^n} \left[\frac{1}{2}\sum_{i<j}w_{ij}\left(1-(2x_i-1)(2x_j-1) \right): x\le w\right]\\
&=&\max\limits_{x\in\mathbb{B}^n} \left[\sum_{i<j}w_{ij}\left(x_i+x_j-2x_ix_j\right): x\le w\right]\\
&=&\max\limits_{x\in\mathbb{B}^n} \left[(W\mathbf{1})^Tx-x^TWx: x\le w\right],
\end{eqnarray*}
where the non-negative edge weight function $w: E\mapsto \mathbb{R}_{+}$ is defined over an undirected graph $(V=[n],E=[n]\times [n])$, and $W=(w_{ij})_{n\times n}\in\mathbb{R}^{n\times n}_{+}$ such that $W=W^T$ and $w_{ii}=0, i\in[n]$. For vertex $i\in [n]$ as a player, $x_i$ is a binary decision variable indicating whether asset $i$ is on one side of the selected cut or not. 

Let $q:=\mathbf{0}_n$. Then Proposition~\ref{prop:quad-game} implies that $\nu$ has a non-empty core if and only if for all distinct pair $\forall i, j\in \left\{k\in [n]: (W\mathbf{1}_n)^T\mathbf{e}_k>0\right\}$, we have $w_{ij}=0$, which implies $W=0$ because $W\ge 0$. Hence $\nu$ always has an empty core unless $G$ is an empty graph (i.e., edgeless), in which case, the game is trivially $\nu(w)=0$ and hence a unique trivial core member $y=(W\mathbf{1}_n)^{+}=0$.

Since the core is always empty, we consider approximate core via Theorem~\ref{thm:appro-core}.
\begin{fact} The smallest $\gamma$ such that the $\gamma$-core is non-empty for the maximum cut game is $4$.
\end{fact}
\begin{proof}
\[
\frac{\nu_{\mathbb{R}^n_{+}, F}\left(\mathbf{1}_n\right)}{\nu_{\mathbb{B}^n, f}\left(\mathbf{1}_n\right)}=\frac{\mathbf{1}_n^TW\mathbf{1}_n}{\nu_{\mathbb{B}^n, f}\left(\mathbf{1}_n\right)}\le\frac{\mathbf{1}_n^TW\mathbf{1}_n}{\frac{1}{4} \mathbf{1}_n^TW\mathbf{1}_n}=4,
\]
where the inequality follows because the random partition argument implies that
\begin{eqnarray*}
\nu_{\mathbb{B}^n, f}\left(\mathbf{1}_n\right)&=&\max\limits_{x\in\mathbb{B}^n} \left[(W\mathbf{1})^Tx-x^TWx \right]\ge \frac{1}{4} \mathbf{1}_n^TW\mathbf{1}_n.
\end{eqnarray*}
Therefore, Theorem~\ref{thm:appro-core} implies that 4-core is always non-empty with the unique $4$-core member $y=W\mathbf{1}\in 4-\text{core}(\nu_{\mathbb{B}^n, f})$. 

This is the smallest $\gamma$ such that $\gamma$-core exists. Consider the complete graph $K_n$ with even $n$ and uniform edge weights. Then $\mathbf{1}_n^TW\mathbf{1}_n=n(n-1)$, while the maximum cut capacity $\frac{n^2}{4}$ achieved by partitioning the graph equally. So $\lim\limits_{n\to\infty} \frac{\frac{n^2}{4}}{n(n-1)}=\frac{1}{4}$. 
\end{proof}

\end{enumerate}

\subsubsection{Quadratic matching game}

We consider a more complicated $A$ than the identity matrix in this section. Given a (undirected) graph $G=(V=[n], E=[m])$, let $A \in \mathbb{B}^{n\times m}$ be the incidence matrix of $G$ such that $A_{ij}=1$ if vertex $i\in V$ is incident with edge $j\in E$; and 0 otherwise. The characteristic function of the quadratic matching game with vertices as players is defined as follows:
\begin{eqnarray*}
\mathbb{B}^n\ni w\mapsto \nu_{\mathbb{B}^m,f}(w):=\max_{x \in \mathbb{B}^m} \left[f(x):=b^T x+x^TQx: Ax\le w\right],
\end{eqnarray*}
where the binary decision variable $x_i$ indicates whether edge $i\in [m]$ is included in the matching; $b\in\mathbb{R}^m$; and $Q\in\mathbb{R}^{m\times m}$ with $Q^T=Q, q_{ii}=0, i\in [m], q_{ij}\le 0, i, j\in [m]$, and hence the objective function is individually subadditive (or equivalently submodular).

From earlier discussion, the core for this quadratic matching game is non-empty if and only if
\begin{eqnarray*}
\max_{x\in\mathbb{B}^m} \left[b^Tx: Ax\le \mathbf{1_n}, q_{ij}x_ix_j=0, \forall i\ne j\in [m]\right]=\max_{x\in\mathbb{R}^m_{+}} \left[b^Tx: Ax\le \mathbf{1_n}\right]
\end{eqnarray*}

While finding a core member is polynomially solvable whenever non-empty, however, we now show that checking core non-emptiness for this quadratic matching game is NP-complete even for bipartite graph, in contrast to the counter-part linear game where testing core non-emptiness for bipartite graph is polynomially solvable~\citep[Corollary 3]{deng1999algorithmic}.

We define conflict graph $G^{\prime}=(E,E')$ where $E'=\{(i,j): q_{ij}=0,~i,j\in E\}$ be the edge pair without conflicts. Let $V_1,\ldots,V_k$ be the vertices of maximal cliques of $G^{\prime}$. If $x$ is a solution of $\nu_{\overline{\mathbb{B}^m},F}\left(\mathbf{1}_n\right)$, then $x$ only use edges corresponding to some clique $V_i$. Thus, we have the following Proposition.

\begin{prop}~\label{prop:q-matching-core}
The individually subadditive quadratic matching game $\nu$ has a non-empty core if and only if
$$
\nu_{\mathbb{R}^m_{+},F}\left(\mathbf{1}_n\right)=\max\{v_1,\ldots,v_k\},
$$
where $v_i$ is the weight of maximum matching of $G[V_i]$ with respect to weight $b$.
\end{prop}

This proposition above suggests an algorithm to check the non-emptiness of the core when the maximal cliques of $G^{\prime}$ are known. Unfortunately, it is well-known that the problem of finding the maximum clique of a graph is NP-hard. We next show that deciding whether the core of the quadratic matching game is non-empty is NP-hard.
Our proof is based on the complexity of the assignment problem with conflicts~\cite{darmann2011paths}.
Conflicts is defined as a series of pair of edges.
For each conflicting pair, at most one edge can occur in a matching.
We show a reduction from the $(3,B2)$-SAT problem, i.e., a 3-SAT problem where each clause contains exactly three literals, and each variable appears exactly twice as a positive literal and exactly twice as a negative literal.
It is known to be NP-complete.
Let $I$ be an arbitrary instance of $(3,B2)$-SAT with $k$ clauses $C_j$ and $n$ variables $x_i$.
We define a bipartite graph $G_I$ in the following way:

\begin{figure}[H]
    \centering
    \includegraphics[width=10cm]{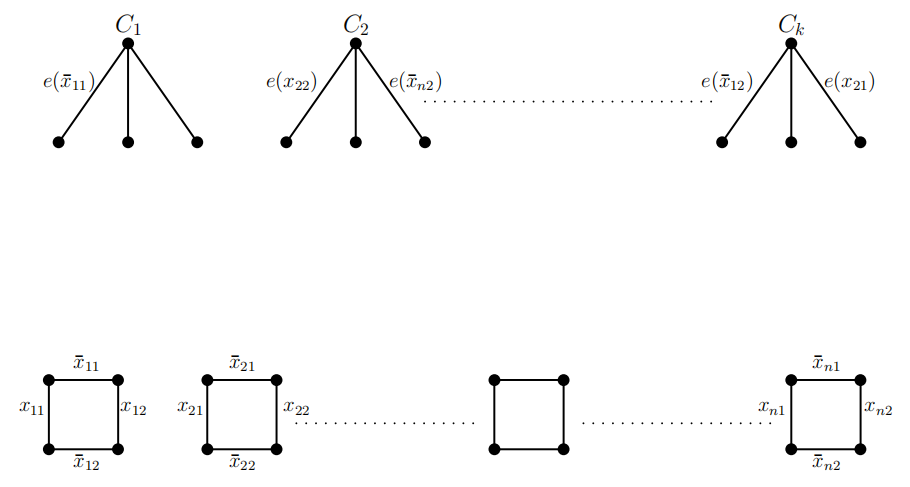}
    \caption{Graph $G$.}
    \label{fig:NP-hard-graph}
\end{figure}

For each variable $x_i$ we introduce a cycle of length four consisting of edges $x_{i1}, x_{i2}, \bar{x}_{i1}$ and $\bar{x}_{i2}$ such that $x_{i1}$ and $x_{i2}$ are not adjacent.
These cycles are isolated components of $G_I$.
Moreover, we introduce for each clause $C_j$ of $I$ an isolated claw rooted at a vertex $C_j$ with the following three edges: If the literal $x_i$ occurs in clause $C_j$ we denote one edge incident to vertex $C_j$ by $e(x_{i1})$ or by $e(x_{i2})$ if the name $e(x_{i1})$ was already used.
If the negated literal $\bar{x}_i$ occurs in clause $C_j$ we denote one edge incident to vertex $C_j$ by $e(\bar{x}_{i1})$ or by $e(\bar{x}_{i2})$ if the name $e(\bar{x}_{i1})$ was already used.

Next, we define the set of conflicts $\mathcal{C}_I$ as the union of $(x_{i1},e(\bar{x}_{i1}))$, $(x_{i2},e(\bar{x}_{i2}))$, $(\bar{x}_{i1},e(x_{i1}))$, $(\bar{x}_{i2},e(x_{i2}))$ for $i=1,2,\ldots,n$.

\begin{lem}~\cite{darmann2011paths}\label{lem:NP-hard}
Given an instance $I$ of $(3,B2)$-SAT, $I$ is TRUE if and only if the size maximum matching of $G_I$ with conflicts $\mathcal{C}_I$ is $2n+k$.
\end{lem}

\begin{prop}\label{prop:assignment-NP-hard}
It is NP-hard to decide whether the core of quadratic matching game is empty or not even if the characteristic function is submodular and the graph is bipartite.
\end{prop}

\begin{proof}
We show a reduction from the $(3,B2)$-SAT problem.
Given an instance $I$ of $(3,B2)$-SAT, we construct a bipartite graph $G_I=(V,E)$ with conflicts $\mathcal{C}_I$ as described above.
Set $b_i=1$ for all $i\in E$, $q_{ij}=-1$ for all $(i,j)\in \mathcal{C}_I$ and $q_{ij}=0$ otherwise.
By Proposition~\ref{prop:q-matching-core}, the core is non-empty if and only if there exists a matching $x$ with size $2n+k$ such that $q_{ij}=0$, $\forall i\neq j\in\text{supp}(x)$.
Thus, $(i,j)\notin \mathcal{C}_I$ for any $\forall i,j\in\text{supp}(x)$, i.e., $x$ satisfies the conflicts $\mathcal{C}_I$.
Therefore, the core is non-empty if and only if $I$ is TRUE by Lemma~\ref{lem:NP-hard}.
\end{proof}

\subsection{Combinatorial ratio games}

We consider combinatorial ratio games in this section. 
\begin{eqnarray}
\label{game:packing-subadditive-ratio}\mathbb{B}^n\ni w\mapsto \nu_{\mathbb{B}^m, f}(w):=\max\limits_{x\in \mathbb{B}^m} \left[f(x):=\frac{c^T x}{d_0+d^T x}: Ax\le w\right]\in  \mathbb{R},
\end{eqnarray}
under the assumptions (a-c), where $d\in \mathbb{R}^m_{+}, c\in \mathbb{R}^m_{++}, d_0\in \mathbb{R}_{++}$, and $X=\mathbb{B}^m$.
Then $f$ is individually subadditive from Proposition~\ref{prop:id-sub}(vi), and its basis-linear relaxation is as follows:
\[
F(x)=\left(\frac{c_1}{d_0+d_1},\ldots,\frac{c_m}{d_0+d_m}\right)^T x. 
\]
Moreover, we have
\begin{eqnarray*}
\nonumber X=\mathbb{B}^m\implies \overline{X}=\overline{\mathbb{B}}^n&=&\left\{x\in \mathbb{B}^m: f(x)=F(x)\right\}\\
\nonumber&=&\left\{x\in\mathbb{B}^m: \frac{c^T x}{d_0+d^T x}=\left(\frac{c_1}{d_0+d_1},\ldots,\frac{c_m}{d_0+d_m}\right)^T x\right\}\\
\nonumber&=&\left\{x\in\mathbb{B}^m: \sum_{i\in[m]}\frac{c_ix_i\left(d^Tx-d_i\right)}{(d_0+d_i)(d_0+d^Tx)}=0 \right\}\\
\nonumber&\overset{x\in\mathbb{B}^n\implies c_ix_i\left(d^Tx-d_i\right)\ge 0}{=}&\left\{x\in\mathbb{B}^m: c_ix_i\left(d^Tx-d_i\right)=0, i\in [m]\right\}\\
\label{eq:ratio-f=F}&\overset{m\ge 2, c\in\mathbb{R}_{++}}{=}&
\left\{x\in\mathbb{B}^m: \sum_{i\in[m]}x_i\le 1 \text{ or } d_ix_i=0, \forall i\in [m]\right\}
\end{eqnarray*}

\begin{prop}\label{prop:ratio-game} Assume that $m\ge 2$, The combinatorial ratio game has a non-empty core if and only if 
\[
\nu_{\mathbb{R}^m_{+},F}=\max\left\{\max\limits_{i\in [m]} \frac{c_i}{d_0+d_i}, \max\limits_{x\in \mathbb{B}^{|K|}} \left[\sum_{k\in K}\frac{c_k}{d_0}x_k: \overline{A}x\le \mathbf{1}_n\right]\right\},
\]
where $K:=\{i\in[m]: d_i=0\}$, and $\overline{A}$ is the matrix obtained from $A$ by removing all columns whose indices are not in $K$ (or equivalently by retaining all columns whose indices are in $K$).
\end{prop}
\begin{proof}
Theorem \ref{thm:core_nonempty} implies that the core of the game $\nu_{\mathbb{B}^m,f}$ is non-empty if and only if 
\begin{eqnarray*}
\nu_{\mathbb{R}^m_{+},F}\left(\mathbf{1}_n\right)=\nu_{\overline{\mathbb{B}}^n, F}\left(\mathbf{1}_n\right)&:=&\max\limits_{x\in \overline{\mathbb{B}}^n} \left[\left(\frac{c_1}{d_0+d_1},\ldots,\frac{c_m}{d_0+d_m}\right)^T x: Ax\le \mathbf{1}_n\right]\\
&=&\max\limits_{x\in \mathbb{B}^m} \left[\left(\frac{c_1}{d_0+d_1},\ldots,\frac{c_m}{d_0+d_m}\right)^T x: Ax\le \mathbf{1}_n, \left(\mathbf{1}_m^Tx\le 1 \text{ or } d_ix_i=0, \forall i\in [m]\right)\right]\\
&=&\max\left\{\max\limits_{x\in \mathbb{B}^m} \left[\left(\frac{c_1}{d_0+d_1},\ldots,\frac{c_m}{d_0+d_m}\right)^T x: Ax\le \mathbf{1}_n, \mathbf{1}_m^Tx\le 1\right]\right.,\\
&&\left.\max\limits_{x\in \mathbb{B}^m} \left[\left(\frac{c_1}{d_0+d_1},\ldots,\frac{c_m}{d_0+d_m}\right)^T x: Ax\le \mathbf{1}_n, d_ix_i=0, \forall i\in [m]\right]\right\}\\
&=&\max\left\{\max\limits_{i\in [m]} \frac{c_i}{d_0+d_i}, \max\limits_{x\in \mathbb{B}^{|K|}} \left[\sum_{k\in K}\frac{c_k}{d_0}x_k: \overline{A}x\le \mathbf{1}_n\right]\right\},
\end{eqnarray*}
where the second to the last equality follows from the simple fact that 
\[
\sup_{x\in A\cup B} f(x)=\max\left\{\sup_{x\in A} f(x), \sup_{x\in B} f(x)\right\},
\]
and by default, $\sum_{k\in K=\emptyset}\frac{c_k}{d_0}x_k:=0$.
\end{proof}

Proposition~\ref{prop:ratio-game} allows us to test core non-emptiness in polynomial time whenever the LP $\nu_{\mathbb{R}^m_{+},F}\left(\mathbf{1}_n\right)$ and the linear combinatorial problem $\max\limits_{x\in \mathbb{B}^{|K|}} \left[\sum_{k\in K}\frac{c_k}{d_0}x_k: \overline{A}x\le \mathbf{1}_n\right]$ can be both polynomially solvable because $\max\limits_{i\in [m]} \frac{c_i}{d_0+d_i}$ is easy to find. As concrete examples, core non-emptiness cab be tested in polynomial time for the ratio version of classic combinatorial games, such as the maximum flow game, the matching game, etc. 

This characterization becomes more explicit when we consider special constraint matrix $A$, as will be done below. In particular we consider two classes of applications. The first is when $A$ is the identity matrix $I_{n}$, which models games like the portfolio game and the maximum cut game. The second is when $A$ is the incidence matrix of a graph, which models games like the matching game.

\subsubsection{Assortment game}

We assume that $A=I_{n}$ throughout this section, and hence $m=n$. Then From Proposition~\ref{prop:ratio-game}, the core is nonempty if and only if
\begin{eqnarray*}
\nu_{\mathbb{R}^n_{+},F}\left(\mathbf{1}_n\right)&=&\max\limits_{x\in \mathbb{R}^n_{+}} \left[\left(\frac{c_1}{d_0+d_1},\ldots,\frac{c_n}{d_0+d_n}\right)^T x: x\le \mathbf{1}_n\right]=\sum_{i\in [n]}\frac{c_i}{d_0+d_i}=\\
\nu_{\overline{\mathbb{B}}^n, F}\left(\mathbf{1}_n\right)&=&\max\left\{\max\limits_{i\in [m]} \frac{c_i}{d_0+d_i}, \max\limits_{x\in \mathbb{B}^{|K|}} \frac{c_k}{d_0}\right\}.
\end{eqnarray*}

This characterization implies that $K\ne \emptyset$; otherwise, the core is empty because
\[
\nu_{\overline{\mathbb{B}}^n, F}\left(\mathbf{1}_n\right)=\max\limits_{i\in [m]} \frac{c_i}{d_0+d_i}<\sum_{i\in [n]}\frac{c_i}{d_0+d_i}=\nu_{\mathbb{R}^n_{+},F}\left(\mathbf{1}_n\right). 
\]
Since $K\ne \emptyset$, the characterization becomes 
\[
\sum_{i\in [n]}\frac{c_i}{d_0+d_i}=\sum_{k\in K}\frac{c_k}{d_0}. 
\]
implying that $K=[n]\iff d_1=\dots=d_n=0$, the original problem has a linear objective function. So we have the following full characterization for core non-emptiness.
 \begin{prop}\label{prop:ratio-assortment-game} The assortment game (\ref{game:packing-subadditive-ratio}) when $A=I_n$ ($n\ge 2$) has a non-empty core if and only if and only if the original problem has a linear objective function.
 \end{prop}

As a practical application of this result Proposition~\ref{prop:ratio-assortment-game}, the above game is an assortment game~\footnote{Our game is different from the assortment game with externalities considered in~\citep{nip2022competitive}.} arising from the assortment problem under the multinominal logit model~\citep{talluri2004revenue};
\begin{eqnarray}
\label{game:packing-subadditive-ratio}\mathbb{B}^n\ni w\mapsto \nu_{\mathbb{B}^m, f}(w):=\max\limits_{x\in \mathbb{B}^m} \left[f(x):=\frac{\sum_{i\in [n]}p_ie^{u_i}x_i}{1+\sum_{i\in n}e^{u_i}x_i}: x\le w\right]\in  \mathbb{R},
\end{eqnarray}
where $p\in\mathbb{R}^n_{++}$ is the price vector of $n$ products viewed as players; $u\in\mathbb{R}^n_{++}$ is the utility vector of the $n$ products; and $x\in\mathbb{B}^n$ is the decision binary variable indicating whether products are selected into the assortment or not.  

Proposition~\ref{prop:ratio-assortment-game} says that core-emptiness implies that $e^{u_i}=0\iff u_i=-\infty, i\in [n]$, which may be unrealistic, or said differently, practical assortment game has empty core in general. So we consider approximate core instead via Theorem~\ref{thm:appro-core}.
\begin{fact} The smallest $\gamma$ such that the $\gamma$-core is non-empty for the assortment game is $n$.
\end{fact}
\begin{proof}
Note that the ratio function is quasi-linear and hence there exists an optimal solution at one of its vertices.
\begin{eqnarray*}
\nu_{\mathbb{B}^n, f}\left(\mathbf{1}_n\right)&=&\max_{x\in\mathbb{B}^n}\frac{\sum\limits_{i\in [n]}p_ie^{u_i}x_i}{1+\sum\limits_{i\in [n]}e^{u_i}x_i}=\max_{x\in [0,1]^n}\frac{\sum\limits_{i\in [n]}p_ie^{u_i}x_i}{1+\sum\limits_{i\in [n]}e^{u_i}x_i}=\max_{i\in [n]}\frac{\sum\limits_{j\in [i]}p_je^{u_j}}{1+\sum\limits_{j\in [i]}e^{u_j}},
\end{eqnarray*}
where the labels in the last quantity are sorted such that $p_1\ge\dots\ge p_n$.

So
\[
\frac{\nu_{\mathbb{R}^n_{+}, F}\left(\mathbf{1}_n\right)}{\nu_{\mathbb{B}^n, f}\left(\mathbf{1}_n\right)}=\frac{\sum\limits_{i\in [n]}\frac{p_iv_i}{1+v_i}}{\max\limits_{i\in [n]}\frac{\sum\limits_{j\in [i]}p_jv_j}{1+\sum\limits_{j\in [i]}v_j}}\overset{v_j:=v_j}{=}\frac{\sum\limits_{i\in [n]}\frac{p_iv_i}{1+v_i}}{\max\limits_{i\in [n]}\frac{\sum\limits_{j\in [i]}p_jv_j}{1+\sum\limits_{j\in [i]}v_j}}\overset{p_1\ge\dots\ge p_n}{\le} 
\frac{\sum\limits_{i\in [n]}\frac{\sum\limits_{j\in [i]}p_jv_j}{1+\sum\limits_{j\in [i]}v_j}}{\max\limits_{i\in [n]}\frac{\sum\limits_{j\in [i]}p_jv_j}{1+\sum\limits_{j\in [i]}v_j}}
\le n
\]
Therefore, Theorem~\ref{thm:appro-core} implies that $n$-core is always non-empty with the unique $n$-core member $y_i=\frac{p_ie^{u_i}}{1+e^{u_i}}=\frac{p_i}{1+e^{-u_i}}\in n-\text{core}(\nu_{\mathbb{B}^n, f})$. .

This is the smallest $\gamma$ such that $\gamma$-core exists. Let $p_1=\dots=p_n=p>0$ and $v_1=\dots=v_n=v>0$ Then $\nu_{\mathbb{B}^n, f}\left(\mathbf{1}_n\right)=\max\limits_{i\in [n]}\frac{ipv}{1+iv}=\frac{npv}{1+nv}$, while $\nu_{\mathbb{R}^n_{+}, F}\left(\mathbf{1}_n\right)=\frac{npv}{1+v}$. So $\lim\limits_{v\to\infty} \frac{\frac{npv}{1+v}}{\frac{npv}{1+nv}}=\lim\limits_{v\to\infty} \left[n-\frac{n-1}{1+v}\right]=n$. 
\end{proof}

\subsubsection{Ratio matching game}
We consider a more complicated $A$ than the identity matrix in this section. Given a bi-weighted (undirected) graph $G=(V=[n], E=[m], c, d)$, let $A \in \mathbb{B}^{n\times m}$ be the incidence matrix of $G$ such that $A_{ij}=1$ if vertex $i\in V$ is incident with edge $j\in E$; and 0 otherwise. Each edge is associated with two types of weights $c\in \mathbb{R}^{m}_{++}$ and $d\in \mathbb{R}^{m}_{+}$, where $c$ indicates profit while $d$ indicates cost. We consider a ratio objective function $f(x):=\frac{c^Tx}{d_0+d^Tx}$ reflects the balance between profit and the cost. The characteristic function of the quadratic matching game with vertices as players is defined as follows:
\begin{eqnarray*}
\mathbb{B}^n\ni w\mapsto \nu_{\mathbb{B}^m,f}(w):=\max_{x \in \mathbb{B}^m} \left[f(x):=\frac{c^Tx}{d_0+d^Tx}: Ax\le w\right],
\end{eqnarray*}
where the binary decision variable $x_i$ indicates whether edge $i\in [m]$ is included in the matching; $c\in\mathbb{R}^m_{++}, d_0\in\mathbb{R}^m_{++}$; and $d\in\mathbb{R}^{m}_{+}$.

Let $G^{\prime}=(V^{\prime}, E^{\prime}, c^{\prime},d^{\prime})$ be the graph obtained from $G=(V, E, c,d)$ as follows: $V^{\prime}=V$, $E^{\prime}=E\backslash \{e\in E: d_e>0\}$, $c^{\prime}=c/d_0$, and $d^{\prime}=0$.

Then we have the following Proposition.
\begin{prop}~\label{prop:r-matching-core}
The ratio matching game $\nu$ has a non-empty core if and only if $\nu_{\mathbb{R}^m_{+},F}\left(\mathbf{1}_n\right)$ equals to the larger of $\max\limits_{i\in [m]}\frac{c_i}{d_0+d_i}$ and the maximum weight matching of $G^{\prime}$.
\end{prop}

Since that the weight of maximum matching of $G^{\prime}$ can be solved in polynomial time, we obtain the following Theorem.
\begin{thm}
Whether the core of ratio matching game is empty or not can be determined in polynomial time.
\end{thm}

\section{Relationship with well-known classes of functions}\label{sec:relationships}

\begin{assumption}
We assume all functions $f: X\mapsto \mathbb{R}$ are games in this work, namely, $f\left(\mathbf{0}_n\right)=0$ (grounded). Moreover, $\{\mathbf{0}_m, \mathbf{e}_1,\dots, \mathbf{e}_m\}\subseteq X\subseteq \mathbb{R}^m$.
\end{assumption}

Various classes of utility functions are defined in the literature, where confusion can arise because identical terminology is often used to convey distinct meanings in different fields. Consequently, to facilitate a clear understanding of this work's results, it is necessary to establish precise and unambiguous definitions for these key terms.

\begin{deff}\label{deff:varity} Given a subset $X\subseteq \mathbb{R}^m$, a function $X \ni x\mapsto f(x)\in\mathbb{R}$ is 
\begin{itemize}
\item individually subadditive if $\{\mathbf{0}_m, \mathbf{e}_1,\dots, \mathbf{e}_m\}\subseteq X$, and 
\begin{equation}\label{deff:is}
f(x)\le x_1f(\mathbf{e}_1)+\dots+x_mf(\mathbf{e}_m), \forall x\in X
\end{equation}

\item subadditive, if $\forall x, y\in X: x+y\in X$, and 
\begin{equation}\label{deff:subadd}
f(x+y)\le f(x)+f(y), \forall x, y\in X
\end{equation}
Note that subadditivity implies nonnegativity (by taking $x=y$) but not monotonicity. 
\item grand fractionally subadditive (a.k.a., balanced), if $X=\mathbb{B}^n$, and 
\begin{equation}\label{deff:gfis}
f\left(\mathbf{1}_m\right)=f\left(\sum_{w\in \mathbb{B}^m} w \lambda(w)\right)\le \sum_{w\in \mathbb{B}^m} \lambda(w) f\left(w\right), \forall \lambda \in \mathbb{R}^{2^m}_{+}: \sum_{w\in \mathbb{B}^m} w \lambda(w)=\mathbf{1}_m.
\end{equation}
\item fractionally subadditive (a.k.a, XOS, totally balanced), if $X=\mathbb{B}^m$, and $\forall v\in\mathbb{B}^m$, 
\begin{equation}\label{deff:gfis}
f(v)\le \sum_{w\in \mathbb{B}^m: w\le v} \lambda(w) f\left(w\right), \forall \lambda \in \mathbb{R}^{2^m}_{+}: \sum_{w\in \mathbb{B}^m: w\le v} w \lambda(w)\ge \mathbf{1}_m\odot v.
\end{equation}
Note that FS implies $f(0)=0$ (by taking $v=w=\mathbf{0}_m, \lambda(w)=1$), monotonicity (by taking $v\le w, \lambda(w)=1$) and hence nonnegativity. 
\item submodular, if $X$ is a sublattice of $\mathbb{R}^m$ and 
\begin{equation}\label{deff:gfis}
f(x\vee y)+f(x\wedge y)\le f(x)+f(y), \forall x, y\in X
\end{equation}
Submodularity implies neither non-negativity nor monotonicity.
\item monotone, if 
\[
f(x)\le f(y), \forall x\le y\in X
\]
\end{itemize}
\end{deff}

We have the following relationship for set functions when $X=\mathbb{B}^n$:
\begin{itemize}
\item For any game:
\[
\text{NSM}\subset \text{ GFS } \subset \text{ SA }\subset \text{ IS}
\]
\item For monotone games:
\[
\text{MSM }\subset \text{FS }\subset \text{ GFS } \subset \text{ SA }\subset \text{ IS}
\]
\end{itemize}
where NSM is the set of non-negative submodular functions; FS is the set of grand fractionally subadditive functions; GFS is the set of grand fractionally subadditive functions; and MSM is the set of monotone submodular functions

Note that individual subadditiveness is a strictly weaker property than submodularity, which implies subadditiveness, and hence further implies individual subadditiveness. Here are examples illustrating the properness of these inclusions.

It is easy to see that any bivariate function $f: \mathbb{B}^2\mapsto \mathbb{R}$ with $f(0,0)=$ is subadditive if and only if it is submodular. The following examples have dimension at least three.

\begin{example} 
\quad

\begin{itemize}
\item Individually subadditive, but not subadditive:
\[
\begin{aligned}
f(0,0,0) &= 0, \\
f(1,0,0) &= 2, & f(0,1,0) &= 2, & f(0,0,1) &= 2, \\
f(1,1,0) &= 3, & f(1,0,1) &= 3, & f(0,1,1) &= 3, \\
f(1,1,1) &= 6.
\end{aligned}
\]
$f(1,1,1)=6>5=f(1,0,0)+f(0,1,1)$ implies that it is not subadditive. 
\item Subadditive, but not submodular:
\[
\begin{aligned}
f(0,0,0) &= 0, \\
f(1,0,0) &= 1, & f(0,1,0) &= 1, & f(0,0,1) &= 1, \\
f(1,1,0) &= 1.5, & f(1,0,1) &= 1.5, & f(0,1,1) &= 1.5, \\
f(1,1,1) &= 3.
\end{aligned}
\]
$f(1,1,0)+f(1,0,1)=3<4=f(1\vee 1,1\vee 0, 0\vee 1)+f(1\wedge 1,1\wedge 0, 0\wedge 1)=f(1,1, 1)+f(1,0, 0)$ implies that $f$ is not submodular. 
\end{itemize}
\end{example}

However, these relations may break down for $X$ other than $\mathbb{B}^n$, which will demonstrated for the quadratic function in Proposition~\ref{prop:id-sub}(v), where individual subadditiveness is actually a stronger property than submodularity for continuous domains such as $X=[0,1]^n$. In general the class of submodular functions and the class of individual subadditive functions on $[0,1]^n$ are not included one way or the other, as demonstrated by the following examples.
\begin{example}\label{exp:sub-is-not-included}\quad\\
\begin{itemize}
\item Submodular but not individually subadditive: 
\[
[0,1]^2\ni (x_1, x_2)\mapsto f(x_1,x_2)=\sqrt{x_1}+\sqrt{x_2}-\sqrt{x_1x_2}
\]
It is submodular via rearrangement inequality, but it is not individually subadditive because
\[
f\left(\frac{1}{4}, \frac{1}{4}\right)=\frac{3}{4}>\frac{2}{4}=\frac{1}{4}f(1,0)+\frac{1}{4}f(0,1)
\]

\item Individually subadditive but not submodular: 
\[
[0,1]^2\ni (x_1, x_2)\mapsto f(x_1,x_2)=3(x_1+x_2)-x_1x_2(3-2x_1x_2)
\]
It is not submodular via cross derivative $f_{1,2}(x_1, x_2)=-3+8x_1x_2$, 
but it is individually subadditive because
\[
f\left(x_1, x_2\right)\le 3x_1+3x_2=x_1f(1,0)+x_2f(0,1)=3(x_1+x_2)
\]
\item Both submodular and individually subadditive
\[
[0,1]^2\ni (x_1, x_2)\mapsto f(x_1,x_2)=x_1+x_2-x_1x_2
\]
\end{itemize}
\end{example}

\section{Operations preserving individual subadditiveness (IS)}\label{sec:subadd-operations-preservation}
Recall the concept of individually subadditiveness. Given a subset $\{\mathbf{0}_m, \mathbf{e}_1,\dots, \mathbf{e}_m\}\subseteq X\subseteq \mathbb{R}^m$, a function $X \ni x\mapsto f(x)\in\mathbb{R}$ is \emph{individually subadditive} if it satisfies: 
\begin{equation}\label{assump:subadditive-cont}
f(x)\le x_1f(\mathbf{e}_1)+\dots+x_mf(\mathbf{e}_m), \forall x\in X
\end{equation}

Therefore, individually subadditiveness of a function depends on the domain. For $X=\mathbb{B}^n$, subadditiveness ($f(x+y)\le f(x)+f(y), \forall x, y\in \mathbb{R}^m$) implies individually subadditive, which may not be true for $X=\mathbb{R}^n$. However, homogeneous subadditive function (i.e., sublinear function) are individually subadditive $X=\mathbb{R}^n$. Moreover, submodular functions are individually subadditive functions for both $X=\mathbb{B}^n$ and $X=\mathbb{R}^n$. 

Note that the conjugate dual operation $\overline{\nu}(w):=\nu\left(\mathbf{1}_n\right)-\nu\left(\mathbf{1}_n-w\right)$ widely used in cooperative game theory does not transform superadditive game to subadditive game.

\begin{prop}\label{prop:id-sub} Here are more rules on checking individually subadditiveness.
\begin{enumerate}[(i)]
\item Conic combination and limit preserves individual subadditiveness/superadditiveness; namely the set of all individual subadditive functions is a closed convex cone (ccc).

\item The indicator function $\delta_X$ is individually superadditive for any $\{\mathbf{0}_m, \mathbf{e}_1,\dots, \mathbf{e}_m\}\subseteq X\subseteq \mathbb{R}^m$.

\item Maximization preserves individual subadditiveness (analogously minimization preserve individual superadditiveness):  If $f, g: X\mapsto \mathbb{R}$ are individual subadditive, the $h:=\max\{f, g\}$ is individual subadditive. Analogously, If $f, g: X\mapsto \mathbb{R}$ are individual superadditive, the $h:=\min\{f, g\}$ is individual superadditive. Here is a quick proof. 
\begin{proof} We only show maximum because minimum follows from $f\wedge g=f+g-f\vee g$. 
\begin{eqnarray*}
\forall x\in X: h(x)&=&f(x)\vee g(x)\le\left(\sum_{i\in [n]}x_if(\mathbf{e}_i)\right)\vee \left(\sum_{i\in [n]}x_ig(\mathbf{e}_i)\right)\\
&\le& \sum_{i\in [n]}x_i [f(\mathbf{e}_i)\vee g(\mathbf{e}_i)]=\sum_{i\in [n]}x_i h(\mathbf{e}_i)
\end{eqnarray*}
\end{proof}

\item Linear transformation: Given two subsets $\{e_1,\dots, e_m\}\subseteq X\subseteq \mathbb{R}^m$ and $Y\subseteq \mathbb{R}^n$, a function $Y\ni y\mapsto f(y)\in\mathbb{R}$, and a matrix $A\in\mathbb{R}^{n\times m}$ such that $Ax\in Y, \forall x\in X$, define 
\[
X\ni x\mapsto g(x):=f(Ax)
\]
Then $f$ is $(x, Y)$-subadditive (namely, $f(x_1y^1+x_2y^2)\le x_1f(y_1)+x_2f(y_2), \forall x=(x_1, x_2)\in X, y_1, y_2\in Y$) implies that $g$ is individually subadditive.
\begin{proof}
\begin{eqnarray*}
\forall x\in X: g(x) &=&f(Ax)=f\left(\sum_{i\in [m]}x_iA\mathbf{e}_i\right)\\
&\overset{f \text{ $(x, Y)$-subadditive}}{\le}& \sum_{i\in [m]}x_if\left(A\mathbf{e}_i\right)=\sum_{i\in [m]}x_i g(\mathbf{e}_i)
 \end{eqnarray*}
\end{proof}

Note that when $X=\mathbb{B}^m$, then $(x, Y)$-subadditive is just subadditive for set functions. However, for $X=\mathbb{R}^m_{+}$, then $(x, Y)$-subadditive is not subadditive (which does not require positive homogeneity) for regular functions, but equivalent to sublinear for any vector space $X$ because any two properties among subadditivity, convexity, and positive homogeneity implies the other other.

\item Quadratic function: $X\ni x\mapsto f(x)=b^Tx+x^TQx\in\mathbb{R}$, where $\{\mathbf{0}_m, \mathbf{e}_1,\dots, \mathbf{e}_m\}\subseteq X\subseteq \mathbb{R}^n_{+}$, $b\in\mathbb{R}^n$ and symmetric matrix $Q\in\mathbb{R}^{n\times n}$. Denote $q:=(q_{11},\dots, q_{nn})^T=(\mathbf{e}_1^TQ\mathbf{e}_1,\dots, \mathbf{e}_n^TQ\mathbf{e}_n)^T$. Then 
\begin{enumerate}[(a)]
 \item $X=\mathbb{B}^n$. $f$ is individually subadditive iff $q_{ij}\le 0, i\ne j$; namely $f$ is submodular.
\item $X=[0,1]^n$. $f$ is individually subadditive iff $q_{ii}\ge 0, i\in [n], q_{ij}\le 0, i\ne j$; namely $f$ is submodular and $Q$ has non-negative diagonal entries.
\item $X=\mathbb{R}^n_{+}$. $f$ is individually subadditive iff $q_{ii}=0, i\in [n], q_{ij}\le 0, i\ne j$; namely $f$ is submodular and $Q$ has zero diagonal entries (a special DR-submodular function~\citep{marinacci2005ultramodular,bian2017continuous}, i.e., the multilinear extension of some submodular set function).
\item $X=\mathbb{R}^n$. $f$ is individually subadditive iff $Q\equiv 0$; namely $f$ is a linear function.
\end{enumerate}
\begin{proof} The basis-linear relaxation of $f$ is
\[
F(x)=b^Tx+\sum_{i\in [n]}x_i\mathbf{e}_i^TQ\mathbf{e}_i=(b+q)^Tx,
\]
 Therefore, our objective is to characterize the matrix $Q$ such that
\begin{eqnarray*}
\delta:=\inf_{x\in X}[F(x)-f(x)]&=&\inf_{x\in X}\left[q^Tx-x^TQx\right]\\
&=&\inf_{x\in X}\left[q^Tx\odot (I_n-x)-x^T\left[\left(I_n-\mathbf{1}_n\mathbf{1}_n^T\right)\odot Q\right]x\right]\\
&=&\inf_{x\in X} \left[(x^T, 1)\begin{pmatrix}
-Q&\frac{q}{2}\\
\frac{q}{2}&0
\end{pmatrix}\begin{pmatrix}x\\ 1\end{pmatrix}\right]\ge 0
\end{eqnarray*}

For (a), $x\in\mathbb{B}^n\implies q^Tx\odot (I_n-x)=0$, and hence
\begin{eqnarray*}
\delta=\min_{x\in\mathbb{B}^n}\left[-2\sum_{i<j}q_{ij}x_ix_j\right]
\end{eqnarray*}
Sufficiency being evident, for necessity, choose $\bar{x}=\mathbf{e}_i+\mathbf{e}_j, \forall i\ne j$, we have
\begin{eqnarray*}
0\le \delta\le F(\bar{x})-f(\bar{x})=-2q_{ij}\implies q_{ij}\le 0
\end{eqnarray*}

For (b), 
\begin{eqnarray*}
\delta&=&\min_{x\in [0,1]^n}\left[\sum_{i\in [n]}q_{ii}x_i(1-x_i)-2\sum_{i<j}q_{ij}x_ix_j\right]
\end{eqnarray*}
Sufficiency being evident, for necessity, $q_{ij}\le 0$ follows from the same argument as that in (a). To show $q_{ii}\ge 0$, choose $\bar{x}=t\mathbf{e}_i, i\in [n], t\in [0,1]$, we have 
\begin{eqnarray*}
0\le \delta \le F(\bar{x})-f(\bar{x})=t(1-t)q_{ii}, \forall t\in [0,1]\implies q_{ii}\ge 0.
\end{eqnarray*}

For (c), 
\begin{eqnarray*}
\delta&=&\min_{x\in \mathbb{R}^n_{+}}\left[\sum_{i\in [n]}q_{ii}x_i(1-x_i)-2\sum_{i<j}q_{ij}x_ix_j\right]
\end{eqnarray*}
Sufficiency being evident, for necessity, $q_{ij}\le 0$ follows from the same argument as that in (a). To show $q_{ii}=0$, choose $\bar{x}=t\mathbf{e}_i, i\in [n], t\ge 0$, we have 
\begin{eqnarray*}
0\le \delta \le F(\bar{x})-f(\bar{x})=t(1-t)q_{ii}, \forall t\ge 0\implies q_{ii}= 0.
\end{eqnarray*}

For (d)
\begin{eqnarray*}
\delta&=&\min_{x\in \mathbb{R}^n_{+}}\left[\sum_{i\in [n]}q_{ii}x_i(1-x_i)-2\sum_{i<j}q_{ij}x_ix_j\right]
\end{eqnarray*}
Sufficiency being evident, for necessity, choose $\bar{x}=t\mathbf{e}_i, i\in [n], t\in\mathbb{R}$, we have 
\begin{eqnarray*}
0\le \delta \le F(\bar{x})-f(\bar{x})=t(1-t)q_{ii}, \forall t\in\mathbb{R}\implies q_{ii}= 0.
\end{eqnarray*}
Moreover, choose $\bar{x}=\mathbf{e}_i\pm\mathbf{e}_j, i\ne j\in [n]$, we have
\begin{eqnarray*}
0\le \delta \le F(\bar{x})-f(\bar{x})=2\pm q_{ij} \implies q_{ij}= 0.
\end{eqnarray*}
\end{proof}
Comments:
\begin{itemize}
\item In (c) and (d), both $\mathbb{R}^n_{+}$ and $\mathbb{R}^n$ are special closed convex cone. More general cones can be considered~\citep{hiriart2010variational}.
\end{itemize}

\item Ratio of two monotone modular functions: A ratio set function 
\[
\mathbb{B}^n\ni x\mapsto f(x):=\frac{c^Tx}{d_0+d^Tx}\in\mathbb{R}, 
\]
is individually subadditive, for any $c, d\in\mathbb{R}^n_{+}, d_0>0$. 
We provide two proofs, the first one is directly showing the individual subadditiveness and the second one  showing the stronger property of grand fractionally subadditive, which implies individual subadditiveness. 
\begin{proof} \quad\\ 
Proof 1: 
\[
F(x)-f(x)=\sum_{i\in [n]}\frac{c_ix_i(d^Tx-d_i)}{(d_0+d_i)(d_0+d^Tx)}=\frac{1}{d_0+d^Tx}\sum_{i\in [n]}\frac{c_ix_i(d^Tx-d_i)}{d_0+d_i}\overset{x\in\mathbb{B}^n}{\ge} 0
\]
Proof 2: For any balanced weights $\lambda\in\mathbb{B}^{2^n}$, we have
\begin{eqnarray*}
\sum_{w\in\mathbb{B}^n}\lambda(w)f(w)-f(\mathbf{1})&=&\sum_{w\in\mathbb{B}^n}\lambda(w)f(w)-\frac{\mathbf{1}^Tc}{d_0+\mathbf{1}^Td}\\
&\overset{\lambda \text{ balanced}}{=}&\sum_{w\in\mathbb{B}^n}\lambda(w)f(w)-\frac{\sum_{w\in\mathbb{B}^n}\lambda(w)w^Tc}{d_0+\mathbf{1}^Td}\\
&\overset{c^Tw=f(w)(d_0+d^Tw)}{=}&\sum_{w\in\mathbb{B}^n}\lambda(w)f(w)-\frac{\sum_{w\in\mathbb{B}^n}\lambda(w)f(w)(d_0+d^Tw)}{d_0+\mathbf{1}^Td}\\
&=&\sum_{w\in\mathbb{B}^n}\lambda(w)f(w)\left(1-\frac{d_0+w^Td}{d_0+\mathbf{1}^Td}\right)\overset{w\le \mathbf{1}}{\ge} 0\\
\end{eqnarray*}
\end{proof}

Comments
\begin{enumerate}[(i)]

\item For $x\in [0,1]^n$: the proofs above break down and $f$ is no longer individually subadditive. Here is a quick counter-example: $n=1$, $c_1=d_1=d_0=1$. Then $f(x)=\frac{x}{1+x}$ is individually superadditive rather than individually subadditive because
\[
f(x)-F(x)=f(x)-xf(1)=\frac{x(1-x)}{2(1+x)}\overset{x\in [0,1]}{\ge} 0
\]
This implies that univariate ratio function over the binary domain $\mathbb{B}$ is actually individually additive. 

\item $f$ is not balanced (i.e., not fractionally subadditive): $m=2, c=(1,1), d=(1,2)$.
\item $f$ is submodular for $m\le 2$, but non-submodular for $m\ge 3$: Here is a quick counter example: $m=3, c=(7,1,1), d=(1,1,1)$. 
\[
f(1,1,0)+f(1,0,1)=4+4=8<10=3+7=f(1,1,1)+f(1,0,0).
\]

\end{enumerate}

\item Parametric optimization: Given $c\in \mathbb{R}^{m}_{+}$, and $S\subseteq \mathbb{B}^{m\times n}$, assume that $S$ satisfies $\left(\mathbf{0}_m, \mathbf{0}_n\right)\in S$, and 
\[
\forall \left(y(\mathbf{e}_i), \mathbf{e}_i\right)\in S, i\in [n], \forall (y, x)\in S: y\ge \sum_{i\in [n]} y(\mathbf{e}_i)x_i. 
\]
Then $g$ below is individually superadditive. 
\[
\mathbb{B}^n\ni x\mapsto g(x):=\min_{y\in\mathbb{B}^m}[c^Ty: (y, x)\in S]
\]
\begin{proof}
Let $(y(x), x)$ and $\left(y\left(\mathbf{e}_i\right), \mathbf{e}_i\right)\in S, i\in [n]$ such that $y(x)$ and $y\left(\mathbf{e}_i\right)$ are the optimal solutions corresponding to $x$ and $\mathbf{e}_i\in \mathbb{B}^n$, respectively. Then 
\begin{eqnarray*}
g(x)&=&c^Ty(x)\ge \sum_{i\in [n]} c^Ty\left(\mathbf{e}_i\right)x_i=\sum_{i\in [n]}g\left(\mathbf{e}_i\right)x_i
\end{eqnarray*}
\end{proof}

\end{enumerate}
\end{prop}

\section{Concluding remarks and open questions}\label{sec:conclusion}

To the best of our knowledge, this systematic investigation of nonlinear games induced by subadditive objective functions is the first to address nonlinearity in both the objective function and the constraints. This framework provides new tools to solve games previously considered intractable. We have demonstrated the utility of these methodologies by fully characterizing the core for some specific instances, particularly nonlinear combinatorial games involving quadratic and ratio objective functions over Boolean lattices. Therefore, further studies on nonlinear games other than quadratic and ratio objective functions $f$ and more complex constraint sets $X$ are full of opportunities for our newly developed tools to shine. 

Nevertheless, many games remain outside the scope of our current framework. The following examples highlight areas that warrant further investigation.

\begin{enumerate}[(i)]
\item A full characterization of the core is needed for the following game where $f$ is individual superadditive (e.g., $f$ is supermodular) rather than individual subadditive:
\begin{eqnarray*}
\mathbb{B}^n\ni w\mapsto \nu(w)&:=&\max_{x\in X\subseteq \mathbb{R}^m_{+}} \left[f(x): Ax\ge w\right]\in  \mathbb{R}
\end{eqnarray*}

Hence, this game can be viewed as the "opposite" game of (\ref{game:packing-subadditive}). Some sufficient conditions for core non-emptiness under further model restrictions are known. For instance, $\nu$ is supermodular when the constraint set is a sublattice and the objective function is supermodular ~\citep[Theorem 2.7.6]{Topkis1998}. This supermodularity is a sufficient condition for the non-emptiness of the core.

\item A full characterization of the core is needed for the following LP production game, viewed as a cost game rather than a revenue game. Or in the dual program, the players' information $w$ appears in the objective function rather than in the right-hand side, in contrast to the covering game (Theorem~\ref{thm:nonlinear-covering-game}):
\begin{eqnarray*}
\mathbb{B}^n\ni w\mapsto \nu(w)&:=&\max_{x\in \mathbb{R}^m_{+}} \left[c^Tx: Ax\ge b(w)\right]\in  \mathbb{R}\\
&\overset{\text{LP duality}}{=}&\min_{y\in \mathbb{R}^m_{+}} \left[b(w)^Ty: A^Ty\le c\right]\in  \mathbb{R}
\end{eqnarray*}
where matrix $A\in \mathbb{R}^{n\times m}, c\in \mathbb{R}^{m}$, and $b(w)=(b_1(w), \dots, b_n(w))^T$ where each game $b_i(w): \mathbb{B}^{n}\mapsto \mathbb{R}$ is grounded, namely, $b_i\left(\mathbf{0}_n\right)=0, i\in [n]$.

Consequently, the game defined above differs from the models in \citep{owen1975core,granot1986generalized}, which are formulated as \emph{revenue} games. The main result in those works is that the balancedness of the component games $b_i(w), i\in [n]$ imply the same for $v(w)$. However, this type of games also appears in many applications such as the the economic lot-sizing  (cost) game \citep{chen2016duality}. For specific choices of $A$ and $b(w)$, \citep{chen2016duality} demonstrate that even when the core is non-empty, not all optimal dual solutions necessarily belong to it. On the positive side, another relevant result is that $\nu$ is submodular when the constraint set is a sublattice and the objective function is submodular in $(w, x)$~\citep[Theorem 2.7.6]{Topkis1998}. This submodularity serves as a sufficient condition for the non-emptiness of the core.

Note that submodular set functions (a subclass of individual subadditive functions) are neither convex nor concave, yet they possess structural properties of both. Consequently, both submodular minimization and maximization are meaningful depending on the application. This is analogous to linear functions, which are simultaneously convex and concave, rendering both minimization and maximization well-posed. The current question follows this spirit.

\end{enumerate}

Another promising avenue for future research is the study of the approximate core for nonlinear games. 
Theorem~\ref{thm:appro-core} also offers a foundation for investigating such games, expanding a field that has historically focused on linear game frameworks.

\bibliographystyle{apalike}
\bibliography{Submodular}

\section{Appendix: proofs}

\subsection{Appendix A: Proof of Proposition~\ref{prop:relax-exten}}

\begin{proof}
No matter relaxation or extension, we always have the following relationship:
\begin{eqnarray}\label{eq:relax-exten-1}
F(y^*)\overset{y^*\in\argmax\limits_{y\in X_2} F(y), x^*\in X_1\subseteq X_2}{\ge} F(x^*)
\overset{F|_{X_1}\ge f}{\ge} f(x^*).
\end{eqnarray}
If $F(y^*)=f(x^*)$, then we have equalities throughout, in particular, $F(y^*)=F(x^*)\implies x^*\in \argmax\limits_{y\in X_2} F(y)$; namely when both programs have the same optimal value, then any optimal solution of the original program is also an optimal solution of the relaxed problem. This shows (i)(a) and the ``if" part of (ii)(a). 

If $y^* \in X_1$, then $f(x^*)\ge (y^*)$, together with (\ref{eq:relax-exten-1}), show (i)(b). For an example where it may happen that $y^*\notin \argmax\limits_{x\in X_1} f(x)$, see Example~\ref{exp:int-imlies-no-non-empty-core} in Section~\ref{subsec:nl-game-main-result}.

To show (i)(c), we have  
\begin{eqnarray}\label{eq:relax-relax-2}
f(x^*)\overset{x^*\in\argmax\limits_{x\in X_1} f(x), y^*\in X_1}{\ge} f(y^*)=F(y^*).
\end{eqnarray}
(\ref{eq:relax-exten-1})-(\ref{eq:relax-relax-2}) together implies equalities throughout, and hence $F(y^*)=f(y^*)=f(x^*)=F(x^*)$. In particular, $f(y^*)=f(x^*)$, implying that if an optimal solution to the relaxed problem is also a feasible solution to the original problem and their optimal values are the same, then it is necessarily an optimal solution to the original problem.

From now now, assume $F$ is an extension; namely $F|_{X_1}=f$. 

If $x^*\in\argmax\limits_{y\in X_2} F(y)$, then 
\[
f(x^*)\overset{F|_{X_1}=f}{=}F(x^*)\overset{x^*, y^*\in\argmax\limits_{y\in X_2} F(y)}{=}F(y^*),
\]
which shows the ``only if" part of (ii)(a).

If $y^*\in X_1$, then we have further 
\begin{eqnarray}\label{eq:relax-exten-2}
f(x^*)\overset{x^*\in\argmax\limits_{x\in X_1} f(x), y^*\in X_1}{\ge} f(y^*)\overset{F|_{X_1}=f}{=}F(y^*).
\end{eqnarray}
(\ref{eq:relax-exten-1})-(\ref{eq:relax-exten-2}) Together implies equalities throughout, and hence $F(y^*)=f(y^*)=f(x^*)=F(x^*)$. In particular, $f(y^*)=f(x^*)$, implying that if an optimal solution to the extended problem is also a feasible solution to the original problem, then it is necessarily an optimal solution to the original problem. This shows (ii)(b). 

\end{proof}

\subsection{Appendix B: Proof of Theorem~\ref{thm:nonlinear-packing-const-X-generators}}

\begin{proof}
We first show the following two sets are the same. 
\begin{eqnarray*}
\Omega_1&:=&\left\{y\in\mathbb{R}^n_{+}: a^Ty\ge \nu_{X, f}(a), \forall a\in\mathbb{B}^n\right\}\\
&=&\left\{y\in\mathbb{R}^n_{+}: a^Ty\ge f(x), \forall x\in \{x\in X: Ax\le a\}, \forall a\in \mathbb{B}^n\right\}\\
\Omega_2&:=&\left\{y\in\mathbb{R}^n_{+}: y^TAQ\ge (f(q_1),\dots, f(q_k))\right\}
\end{eqnarray*}

For one direction, choosing $a=Aq_j\in \mathbb{B}^n, j\in [m]$ implies that $x=q_j\in \{x\in X: Ax\le a=Aq_j\}$ because $\{\mathbf{0}_m, q_1,\dots, q_k\}\subseteq X\subseteq \mathbb{R}^m_{+}$.
\begin{eqnarray*}
y\in \Omega_1&\overset{a=Aq_j, x=q_j, j\in [k]}{\implies}& \left[ (Aq_j)^Ty\ge f(q_j), j\in [k], y\ge 0 \right]\\
&\iff& \left[y^TAQ\ge (f(q_1),\dots, f(q_k)), y\ge 0\right] \implies y\in \Omega_2
\end{eqnarray*}
For the other direction, $\forall a\in \mathbb{B}^n, \forall x\in \{x\in X: Ax\le a\}$:
\begin{eqnarray*}
y\in \Omega_2&\implies& a^Ty\overset{y\ge 0}{\ge} y^TAx\overset{(\ref{eq:x-rep-Q})}{=}
[y^TAQ]\left[Q^{\dagger}x\right]\overset{Q^{\dagger}x\ge 0 \text{ from } (\ref{eq:x-rep-Q}) }{\ge}\\ 
&&\left(f(q_1),\dots, f(q_k)\right)Q^{\dagger}x\overset{\text{Def.}~\ref{deff:indiv-supadd_Q}}{\ge}
f\left(QQ^{\dagger}x\right)=f(x)\\
&\implies& y\in \Omega_1,
\end{eqnarray*}
where we used the fact that $y\ge 0$ due to the monotonicity of the game $\nu$ (However, the nonnegativity of $y$ is not needed for partition game due to $Ax=w$). 

Then we have the following relationships:  
\begin{eqnarray*}
\nonumber \nu_{X, f}\left(\mathbf{1}_n\right)&\le &\min\limits_{y\in\mathbb{R}^{n}_{+}}\left\{\mathbf{1}_n^Ty: a^Ty\ge \nu_{X, f}(a), \forall a\in\mathbb{B}^n\right\}\\
\nonumber&=&\min\limits_{y\in\mathbb{R}^{n}}\left\{\mathbf{1}_n^Ty: y\in \Omega_1\right\}\overset{\Omega_1=\Omega_2}{=}\min\limits_{y\in\mathbb{R}^{n}} \left[\mathbf{1}_n^Ty: y\in \Omega_2\right]\\
\nonumber&\overset{(\ref{eq:NLP-relax-primal-Q-basis})-(\ref{eq:NLP-relax-dual-Q-basis})}{=}&\nu_{\text{cone}(X), F}\left(\mathbf{1}_n\right),
\end{eqnarray*}

So the core of $\nu_{X, f}$ is nonempty if and only if $\nu_{X, f}\left(\mathbf{1}_n\right)=\nu_{\text{cone}(X), F}\left(\mathbf{1}_n\right)$. This proves (i).  Moreover, $\Omega_1=\Omega_2$ implies (ii). 

\end{proof}

\subsection{Appendix C: Proof of Theorem~\ref{thm:nonlinear-packing-const-dep-w}}

\begin{proof} 
We first show the following two sets are the same. 
\begin{eqnarray*}
\Omega_1&:=&\left\{y\in\mathbb{R}^n_{+}: a^Ty\ge \nu_{X, f}(a), \forall a\in\mathbb{B}^n\right\}\\
&=&\left\{y\in\mathbb{R}^n_{+}: a^Ty\ge f(x), \forall x\in \{x\in X(w): Ax\le a\}, \forall a\in \mathbb{B}^n\right\}\\
\Omega_2&:=&\left\{y\in\mathbb{R}^n_{+}: y^TA\ge (f(\mathbf{e}_1),\dots, f(\mathbf{e}_m))\right\}
\end{eqnarray*}

For one direction, choosing $a=A\mathbf{e}_j\in \mathbb{B}^n, j\in [m]$ implies that $x=\mathbf{e}_j\in \{x\in X(w): Ax\le a=A\mathbf{e}_j\}$ because $\mathbf{0}_m\in X(0), \mathbf{e}^m_i\in X(\mathbf{e}^n_i), i\in [n]$.
\begin{eqnarray*}
y\in \Omega_1&\overset{a=A\mathbf{e}_j, x=\mathbf{e}_j, j\in [m]}{\implies}& \left[ (A\mathbf{e}_j)^Ty\ge f(\mathbf{e}_j), j\in [m], y\ge 0 \right]\\
&\iff& \left[y^TA\ge (f(\mathbf{e}_1),\dots, f(\mathbf{e}_m)), y\ge 0\right] \implies y\in \Omega_2
\end{eqnarray*}
For the other direction, $\forall a\in \mathbb{B}^n, \forall x\in \{x\in X: Ax\le a\}$:
\begin{eqnarray*}
y\in \Omega_2&\implies& \left[a^Ty\overset{y\ge 0}{\ge} x^T(A^Ty)\overset{x\ge 0}{\ge} x^T(f(\mathbf{e}_1),\dots, f(\mathbf{e}_m))^T=F(x) \ge f(x)\right]\\
&\implies& y\in \Omega_1,
\end{eqnarray*}
where we used the fact that $y\ge 0$ due to the monotonicity of the game $\nu$ (However, the nonnegativity of $y$ is not needed for partition game due to $Ax=w$). 

Then we have the following relationships:  
\begin{eqnarray*}
\nonumber \nu_{X, f}\left(\mathbf{1}_n\right)&\le &\min\limits_{y\in\mathbb{R}^{n}_{+}}\left\{\mathbf{1}_n^Ty: a^Ty\ge \nu_{X, f}(a), \forall a\in\mathbb{B}^n\right\}\\
\nonumber&=&\min\limits_{y\in\mathbb{R}^{n}}\left\{\mathbf{1}_n^Ty: y\in \Omega_1\right\}\\
&\overset{\Omega_1=\Omega_2}{=}&\min\limits_{y\in\mathbb{R}^{n}} \left[\mathbf{1}_n^Ty: y\in \Omega_2\right]\\
\nonumber&\overset{\text{LP dual}}{=}&\max\limits_{x\in \mathbb{R}^m_{+}} \left[(f(\mathbf{e}_1),\dots, f(\mathbf{e}_m))x: Ax\le \mathbf{1}_n\right]\overset{(\ref{eq:NLP-relax-primal-const-dep-w})-(\ref{eq:NLP-relax-dual-const-dep-w})}{:=}\nu_{\mathbb{R}^m_{+}, F}\left(\mathbf{1}_n\right),
\end{eqnarray*}

So the core of $\nu_{X, f}$ is nonempty if and only if $\nu_{X, f}\left(\mathbf{1}_n\right)=\nu_{\mathbb{R}^m_{+}, F}\left(\mathbf{1}_n\right)$. This proves (i).  $\Omega_1=\Omega_2$ implies (ii). Proposition~\ref{prop:relax-exten}(i)(a) implies (iii).  

\end{proof}

\subsection{Appendix D: Proof of Theorem~\ref{thm:nonlinear-packing-rhs-arb}}
\begin{proof}
We show the following two sets are the same:
\begin{eqnarray*}
\Omega_1&:=&\left\{y\in\mathbb{R}^n_{+}: w^Ty\ge \nu_{X, f}(w), \forall w\in \mathbb{B}^n\right\}\\
&=&\left\{y\in\mathbb{R}^n_{+}: w^Ty\ge f(x), \forall x\in \{x\in X: Ax\le bw\}, \forall w\in \mathbb{B}^n\right\}\\
\Omega_2&:=&\left\{y\in\mathbb{R}^n_{+}: y^TA\ge [(f(b\mathbf{e}_1) ,\dots, f(b\mathbf{e}_m))]\right\}
\end{eqnarray*}
For one direction:
\begin{eqnarray*}
y\in \Omega_1&\overset{x=b\mathbf{e}_j, w^j=A\mathbf{e}_j, j\in [m]}{\implies}& \left[y^TA\ge [(f(b\mathbf{e}_1),\dots, f(b\mathbf{e}_m))], y\ge 0 \right]\\
&\implies& y\in \Omega_2
\end{eqnarray*}
For the other direction: 
\begin{small}
\begin{eqnarray*}
y\in \Omega_2&\implies& \left[\forall w\in \mathbb{B}^n, \forall x\in \{x\in X: Ax\le bw\}: w^Ty\overset{y\ge 0}{\ge} b^{-1} x^T(A^Ty)\overset{x\ge 0}{\ge} b^{-1} x^T(f(b\mathbf{e}_1),\dots, f(b\mathbf{e}_m))^T \overset{\text{Def.}~\ref{deff:indiv-supadd_Z^n}}{\ge} f(x)\right]\\
&\implies& y\in \Omega_1
\end{eqnarray*}
\end{small}

Then we have the following relationship;
\begin{eqnarray*}
\nonumber \nu_{X, f}\left(\mathbf{1}_n\right)&\le&\min\limits_{y\in\mathbb{R}^{n}}\left\{\mathbf{1}_n^Ty: y\in\Omega_1\right\}\\
&\overset{\Omega_1=\Omega_2}{=}&\min\limits_{y\in\mathbb{R}^{n}} \left[\mathbf{1}_n^Ty: y\in\Omega_2\right]\\
\nonumber&\overset{y:=bz}{=}&\min\limits_{z\in\mathbb{R}^{n}}\left[\left(b\mathbf{1}_n^T\right)z: A^Tz\ge b^{-1}(f(b\mathbf{e}_1),\dots, f(b\mathbf{e}_m))^T, z\ge 0\right]\\
\nonumber&=&\max\limits_{x\in \mathbb{R}^m_{+}} \left[b^{-1}(f(b\mathbf{e}_1),\dots, f(b\mathbf{e}_m))^Tx: Ax\le b\mathbf{1}_n\right]:=\nu_{\mathbb{R}^m_{+}, F}\left(\mathbf{1}_n\right)
\end{eqnarray*}

So the core of $\nu_{X, f}$ is non-empty if and only if $\nu_{X, f}\left(\mathbf{1}_n\right)=\nu_{\mathbb{R}^m_{+}, F}\left(\mathbf{1}_n\right)$.  This proves (i). $\Omega_1=\Omega_2$ implies (ii). Proposition~\ref{prop:relax-exten}(i)(a) implies (iii).

\end{proof}

\subsection{Appendix E: Proof of Theorem~\ref{thm:nonlinear-packing-game-obj-constraint}}

\begin{proof} We establish that the following two sets are the same: $\Omega_1=\Omega_2$ 
\begin{eqnarray*}
\Omega_1&:=&\left\{y\in \mathbb{R}^n_{+}: w^Ty\ge f(x, w), \forall x\in \{x\in \mathbb{B}^m: Ax\le w\}, \forall w\in \mathbb{B}^n\right\}\\
\Omega_2&:=&\{y\in \mathbb{R}^n_{+}: y^TA\ge \left(f\left(\mathbf{e}_1, A\mathbf{e}_1\right),\dots, f\left(\mathbf{e}_m, A\mathbf{e}_m\right)\right)\},
\end{eqnarray*}

On one direction, we have
\begin{eqnarray*}
y\in\Omega_1&\overset{w=A\mathbf{e}_j, x=\mathbf{e}_j, j\in [m]}{\implies}& [(A\mathbf{e}_j)^Ty\ge f(\mathbf{e}_j, A\mathbf{e}_j), y\ge 0]\\
&\implies& y^TA\ge  \left(f\left(\mathbf{e}_1, A\mathbf{e}_1\right),\dots, f\left(\mathbf{e}_m, A\mathbf{e}_m\right)\right)\\
&\implies& y\in \Omega_2
\end{eqnarray*}

On the other direction: $\forall x\in \{x\in \mathbb{B}^m: Ax\le w\}, \forall w\in\mathbb{B}^n$:
\begin{eqnarray*}
y\in\Omega_2&\implies& w^Ty\ge x^T(A^Ty)\ge x^T\left(f\left(\mathbf{e}_1, A\mathbf{e}_1\right),\dots, f\left(\mathbf{e}_m, A\mathbf{e}_m\right)\right)^T\overset{Def.~\ref{def:matrix-subadditive}}{\ge} f(x, w)\implies y\in \Omega_1
\end{eqnarray*}

Then we have the following relationships:  
\begin{eqnarray*}
\nonumber \nu_{X, f}\left(\mathbf{1}_n\right)&\le&\min\limits_{y\in\mathbb{R}^{n}_{+}}\left\{\mathbf{1}_n^Ty: w^Ty\ge \max\limits_{x\in \mathbb{B}^m}\left[f(x): Ax\le w \right], \forall w\in \mathbb{B}^n\right\}\\
\nonumber&=&\min\limits_{y\in\mathbb{R}^{n}}\left\{\mathbf{1}_n^Ty: y\in \Omega_1\right\}\\
&=&\min\limits_{y\in\mathbb{R}^{n}} \left[\mathbf{1}_n^Ty: y\in \Omega_2\right]\\
\nonumber&\overset{\text{LP dual}}{=}&\max\limits_{x\in \mathbb{R}^m_{+}} \left[(f(\mathbf{e}_1, A\mathbf{e}_1),\dots, f(\mathbf{e}_m, A\mathbf{e}_m))x: Ax\le \mathbf{1}_n\right]:=\nu_{\mathbb{R}^m_{+}, F}\left(\mathbf{1}_n\right),
\end{eqnarray*}

So the core of $\nu$ is nonempty if and only if $\nu_{X, f}\left(\mathbf{1}_n\right)=\nu_{\mathbb{R}^m_{+}, F}\left(\mathbf{1}_n\right)$. This proves (i).  $\Omega_1=\Omega_2$ implies (ii). Proposition~\ref{prop:relax-exten}(i)(a) implies (iii).  
\end{proof}

\subsection{Appendix F: Proof of Theorem~\ref{thm:appro-core}}

\begin{proof} From the proof of Theorem~\ref{thm:nonlinear-packing-game}, we have 
\begin{eqnarray*}
\nu_{\mathbb{R}^m_{+}, F}\left(\mathbf{1}_n\right)&:=&\max\limits_{x\in \mathbb{R}^m_{+}} \left[(f(\mathbf{e}_1),\dots, f(\mathbf{e}_m))x: Ax\le \mathbf{1}_n\right]\\
&\overset{\text{LP duality}}{=}&\min\limits_{y\in\mathbb{R}^{n}} \left[\mathbf{1}_n^Ty: y\in \Omega_2\right]\overset{\Omega_1=\Omega_2}{=}\min\limits_{y\in\mathbb{R}^{n}}\left\{\mathbf{1}_n^Ty: y\in \Omega_1\right\}\\
&=&\min\limits_{y\in\mathbb{R}^{n}_{+}}\left\{\mathbf{1}_n^Ty: a^Ty\ge \nu_{X, f}(a), \forall a\in\mathbb{B}^n\right\}
\end{eqnarray*}

Let 
\[
y^*\in \argmin\limits_{y\in\mathbb{R}^{n}_{+}}\left\{\mathbf{1}_n^Ty: a^Ty\ge \nu_{X, f}(a), \forall a\in\mathbb{B}^n\right\}
\]

For sufficiency, it is easy to verify the following: 
\begin{eqnarray*}
a^Ty^*&\overset{y^* \text{ feasible}}{\ge}& \nu_{X, f}(a), \forall a\in \mathbb{B}^n\\
\mathbf{1}_n^Ty^*&\overset{y^* \text{ optimal}}{=}&\nu_{\mathbb{R}^m_{+}, F}\left(\mathbf{1}_n\right)\le \gamma\nu_{X, f}\left(\mathbf{1}_n\right),
\end{eqnarray*}
implying that $y^*\in \gamma\text{-core}\left(\nu_{X, f}\right)\implies \gamma\text{-core}\left(\nu_{X, f}\right)\ne \emptyset$.

For necessity, if the $\gamma$-core is non-empty, take one member $y\in \gamma\text{-core}\left(\nu_{X, f}\right)$. Then $y$ is a feasible solution of the minimization above. Therefore, 
\[
\nu_{\mathbb{R}^m_{+}, F}\overset{y^* \text{ optimal}}{=}\mathbf{1}_n^Ty^*\overset{y \text{ feasible}}{\le} \mathbf{1}_n^Ty \overset{y \in \gamma\text{-core}\left(\nu_{X, f}\right)}{\le} \gamma\nu_{X, f}\left(\mathbf{1}_n\right). 
\]

\end{proof}
\end{document}